\documentclass[final]{siamart190516}
\usepackage{float}
\usepackage{amsmath}
\usepackage{amsfonts}
\usepackage{amssymb}
\usepackage{fullpage}
\usepackage{cleveref}
\usepackage{siunitx}
\usepackage{mathtools}
\DeclareMathOperator*{\argmax}{arg\,max}

\usepackage{autonum}
\newcommand{\retainlabel}[1]{\label{#1}\sbox0{\ref{#1}}}
\usepackage[notref,notcite,color,final]{showkeys}      % do not show labels
\definecolor{refkey}{named}{blue}
\definecolor{labelkey}{named}{blue}
\makeatletter
\newcommand{\proofstep}[1]{%
  \par% ensure starting on a new paragraph
  \addvspace{\smallskipamount}% some vertical space
  \textit{#1\@addpunct{.}}\enspace\ignorespaces
}%
\makeatother

\newcommand{\overstar}[1]{\overset{\lower.5em\hbox{\scriptsize$\star$}}{#1}}
\newcommand{\sms}[1]{\scriptscriptstyle\text{#1}}
\newcommand{\sss}{\scriptscriptstyle}
\newcommand{\bsb}[1]{\boldsymbol{#1}}
\newcommand{\pd}{\partial}
\newcommand{\td }{\text{d}}
\newcommand{\intG}[1]{\big\langle {#1} \big\rangle_{\sss\Gamma}} 
\newcommand{\intO}[1]{\big\langle {#1} \big\rangle_{\sss\Omega}} 
\newcommand{\tranv}[1]{\left({#1}\right)\tv } 
\newcommand{\tv}{^{\sms{T}}} 
\newcommand{\divs}{\text{div}_{\sms{S}}}

\usepackage{xcolor}
\usepackage{multirow}
\usepackage{booktabs}
\usepackage[font=small,labelfont=bf]{caption}
\usepackage{array}

\title{Shape optimization of slip-driven axisymmetric microswimmers
\thanks{Draft version, \today.
\funding{We acknowledge support from NSF under
grants DMS-1454010 and DMS-2012424. }}}

\author{
Ruowen Liu\thanks{Department of Mathematics, Rider University, Lawrenceville, NJ, USA
  (\email{rliu@rider.edu})}
\and Hai Zhu\thanks{Flatiron Institute, Simons Foundation, New York, NY, USA(\email{hzhu@flatironinstitute.org})}
\and Hanliang Guo\thanks{Departments of Mathematics \& Computer Science, Ohio Wesleyan University, Delaware, OH, USA (\email{hguo@owu.edu})}
\and Marc Bonnet\thanks{POEMS (CNRS, INRIA, ENSTA), ENSTA Paris, 91120 Palaiseau, France
  (\email{mbonnet@ensta.fr})}
\and Shravan Veerapaneni\thanks{Department of Mathematics, University of Michigan, Ann Arbor, MI, USA (\email{shravan@umich.edu})}
}

\begin{document}

\maketitle

% REQUIRED
\begin{abstract}
In this work, we develop a computational framework that aims at simultaneously optimizing the shape and the slip velocity of an axisymmetric microswimmer suspended in a viscous fluid. We consider shapes of a given reduced volume that maximize the swimming efficiency, i.e., the (size-independent) ratio of the power loss arising from towing the rigid body of the same shape and size at the same translation velocity to the actual power loss incurred by swimming via the slip velocity.
% We consider shapes of a given reduced volume that maximize the swimming efficiency, i.e., the (size-independent) ratio of the power loss arising from towing the rigid body of the same shape and size at the same translation velocity to the actual power loss incurred by swimming via the slip velocity. 
The optimal slip and efficiency (with shape fixed) are here given in terms of two Stokes flow solutions, and we then establish shape sensitivity formulas of adjoint-solution that provide objective function derivatives with respect to any set of shape parameters on the sole basis of the above two flow solutions. Our computational treatment relies on a fast and accurate boundary integral solver for solving all Stokes flow problems. We validate our analytic shape derivative formulas via comparisons against finite-difference gradient evaluations, and present several shape optimization examples.
\end{abstract}
\begin{keywords}
Low-Re locomotion, shape sensitivity analysis, integral equations, fast algorithms
\end{keywords}
\begin{AMS}
49M41, 76D07, 65N38
\end{AMS}

%%%% INTRODUCTION %%%%
\section{Introduction}

% Why is swimming of microorganisms important? Why do we care about shape optimizations? What are the applications?
Studying the efficiency of biological microswimmers is pivotal to understanding
natural systems and designing artificial ones for accomplishing various physical tasks~\cite{elgeti2015physics}.
Both the body shape and the locomotory gait contribute to the swimming efficiency of the microswimmers.
%Roughly speaking, a microswimmer achieves high efficiency locomotion via a combination of two things: body shape and gait.  
However, since the inertial effects are negligible at the microscale, optimal swimming strategies markedly diverge from those observed at larger scales (e.g., swimming of fish)~\cite{lauga2009hydrodynamics}. 
Additionally, many microswimmers are covered by densely packed cilia, which are active microtubule-based structures much shorter than the microswimmer's body size~\cite{brennen1977fluid}. 
The periodic beatings of cilia turn the cell surface into an `active slip surface' without much change to the body shape~(see, e.g., \cite{blake1971spherical} and \cite{pedley2016spherical}). As a result, naively finding the swimmer shape to minimize the fluid drag could be a sub-optimal strategy. On the other hand, artificial microswimmers such as phoretic particles locomote by the effective slip velocities on the particle surfaces resulted from the asymmetry of chemical reactions on their surfaces \cite{anderson1989colloid, golestanian2007designing, moran2017phoretic}. Artificial microswimmers have attracted much attention owing to their importance in applications such as targeted drug delivery, microsurgery, and automated transport of cargo/payloads in microfluidic chips \cite{li2017micro}. Consequently, shape optimization for the slip-driven microswimmers can shed light on the shapes and swimming mechanisms of biological microswimmers, and provide guidance for the design and engineering of artificial ones.

%----drag vs swimmer shapes------------------
%The shape of minimal drag in Stokes flow has been studied since 1970s. Specifically, \cite{pironneau1973optimum} and \cite{bourot1974numerical}  found that a pointy  prolate spheroidal shape with semi-cone angle about 60$^\circ$, like an American football, minimizes the fluid drag for a {\em given volume}. Their results show that the minimized fluid drag is roughly 95\% of the fluid drag on the sphere of equal volume.
%In a recent study, \cite{montenegro2015other} found the optimal shape with {\em given surface area} presents sharper tips with semi-cone angle 30.8$^\circ$, and the fluid drag is roughly 89\% of that on the sphere of equal surface area.
%At first sight, these rather modest reductions in drag ($\lesssim 10\%$) may delude us into believing that swimming efficiencies will not be significantly improved by optimizing shapes.
%
%The story, however, is much different if the microswimmer generates a slip velocity on its surface (see, e.g. \cite{leshansky2007frictionless, vilfan2012optimal, guo2021slip}).
%-----------------------------------

In an earlier work, \cite{leshansky2007frictionless} studied the optimal slip velocity of spheroids, using analytical solutions of the Stokes equations in spheroidal coordinates. 
%The slip velocity was expanded in Fourier series whereas its coefficients corresponding to optimal locomotion were determined numerically. A simple closed-form near-optimal slip velocity was also proposed in the paper. 
In contrast to the drag minimization problem where optimal shapes provide marginal efficiency gains over spheroids \cite{pironneau1973optimum, bourot1974numerical, montenegro2015other}, they found that the swimming efficiency grows unbounded with the aspect ratio.
%%%%%%%%
Shortly after, \cite{vilfan2012optimal} optimized the shape and slip velocity for swimming efficiency at the same time, subject to a minimum curvature constraint. Motivated by the cilia carpet that formed the slip surface, the author considered the energy dissipation {\em inside} the cilia carpet, and assumed a local linear relationship between slip velocity and force density, 
which simplified the question significantly to a quadratic problem. 
The optimal shapes with different minimum curvatures evolve from a sphere to a prolate shape with ripples on the surface, and eventually to American football shape with long protrusions from both ends as the allowed minimum curvature decreases from one to almost zero. The activation (slip velocity) of the protruded shape appears to be heavily localized near the ends of the protrusions.

In a recent study~\cite{guo2021slip}, we introduced a numerical algorithm for determining, for a given arbitrary axisymmetric shape, the slip velocity that minimizes the power loss {\em outside} the slip surface while maintaining a given swim speed. By exploiting the quadratic dependence in the slip velocity of the power loss functional, the numerical solution of the resulting optimization problem could be be performed efficiently with the help of a fast boundary integral solver for the forward problem, typically taking only a few seconds on a standard laptop for a given shape~\cite{guo2021slip}. We explored a wide range of shapes with different reduced volume (volume normalized by surface area), and found prolate spheroids to be the most efficient among the tested shapes.

In this work, we develop a computational framework that aims at simultaneously optimizing the shape and the slip velocity while keeping the volume and the surface area constant. While our main focus is on the shape optimization component, we also provide an improved version of our earlier slip optimization method~\cite{guo2021slip} whereby the optimal slip velocity and swim efficiency are obtained for a given fixed shape in terms of two flow solutions. The main contributions of this work are two-fold. First, we establish shape sensitivity formulas tailored to the specific characteristics of the problem at hand. Their derivation exploits the fact that the swimming efficiency is given for any given shape by a Rayleigh quotient of quadratic forms and uses the weak formulation of the flow problems together with reciprocity identities. The resulting shape sensitivities are expressed as integrals on the swimmer boundary involving the two solutions that determine the optimal slip, and in a form consistent with the general structure of shape derivative formulas~\cite{henrot:pierre:18}. They conform to the widely-used adjoint solution approach~\cite{hinze:08}, as they allow to evaluate shape functional derivatives with respect to any chosen set of shape parameters on the basis of only the two aforementioned flow solutions. Second, as in \cite{bonnet2020shape,bonnet2023shape}, we employ boundary integral equation (BIE) techniques to solve the flow problems for any given shape in a straightforward manner. For shape optimization problems, BIEs have the significant advantage of avoiding any volume re-meshing between optimization iterations, on top of other usual advantages over classical domain discretization methods. Moreover, the improved version of the slip optimization component given in this work constitutes an additional contribution, whose role is important for the combined optimization problem at hand since the same two flow solutions provide the optimal slip as well as all shape sensitivities on any given shape. While the combination of adjoint-based methods and BIE methods have been successfully applied to shape optimization problems for Stokes flow previously (e.g., \cite{zabarankin2010three, alouges2011numerical, walker2013analysis, bonnet2020shape, bonnet2023shape}), we are not aware of any work that applied these methods to slip-driven microswimmers.

The paper is organized as follows. We introduce the underlying forward problem in Section~\ref{sc:formulation}, then formulate the optimization problem and derive the sensitivity formulas in Section~\ref{sc:optimization}. The proof of the main shape sensitivity results is then given in Section~\ref{prop:Jw:proof}. We next propose the numerical scheme in Section~\ref{sc:NumMeth} and provide some numerical examples in Section~\ref{sc:results}. Section~\ref{sc:conclusions} closes the paper with concluding remarks.\enlargethispage*{1ex}

%%%% Problem Formulation %%%%
\section{Forward problem formulation} \label{sc:formulation}

\subsection{Geometry and notation}
\label{geometry}

Let the axisymmetric body of the microswimmer occupy the bounded domain $\Omega_{\sms{S}}$ with (closed smooth) boundary $\pd \Omega_{\sms{S}}=\Gamma$, and let $\Omega=\mathbb{R}^3\!\setminus\!\overline{\Omega_{\sms{S}}}$ denote the unbounded fluid region surrounding it. The surface $\Gamma$ is generated by rotating about $\bsb{e}_z$ an open arc $\gamma$ given in the $(\bsb{e}_x, \bsb{e}_z)$-plane by
\begin{equation}
  \gamma\ni\bsb{x}_{\gamma}(t) = R(t)\bsb{e}_x + Z(t)\bsb{e}_z,  \quad 0\le t \le \pi, \label{curve:shape}
\end{equation}
where the parametric interval $t\in[0,\pi]$ is used for consistency with the implementation (see Fig.~\ref{geom:2D}) and $R, Z$ are smooth ($C^2$) functions satisfying
\begin{equation}
  \text{(a) \ }R(0)=R(\pi)=0, \qquad \text{(b) \ }R(t) \geq 0, \qquad
  \text{(c) \ }Z'(0) = Z'(\pi) = 0. \label{eq:rz}
\end{equation}
(the prime symbol $()'$ indicating derivatives of univariate functions). The last condition above ensures that $\Gamma$ is smooth at the poles. The surface $\Gamma$ then has the parametric representation
\begin{equation}
  \Gamma\ni\bsb{x}(t,\phi)
 = R(t)\bsb{e}_r(\phi) + Z(t)\bsb{e}_z, \quad 0\le t\le\pi,\quad 0\le\phi < 2\pi , \label{wall:shape}
\end{equation}
where $\phi$ denotes the angular polar coordinate in the $(\bsb{e}_x,\bsb{e}_y)$-plane and $\bsb{e}_r(\phi)=\bsb{e}_x\cos\phi+\bsb{e}_y\sin\phi$.  The unit tangent vector $\bsb{\tau}(t,\phi)$ in the meridian $(\bsb{e}_r,\bsb{e}_z)$-plane and the unit normal vector $\bsb{n}(t,\phi)$ pointing inwards of the body are given by
\begin{equation}
  \alpha(t)\bsb{\tau}(t,\phi)  = R'(t) \bsb{e}_r(\phi) + Z'(t)\bsb{e}_z, \quad
  \alpha(t)\bsb{n}(t,\phi) = Z'(t)\bsb{e}_r(\phi) - R'(t)\bsb{e}_z, \label{tn:axi}
\end{equation}
where $\alpha(t) := \sqrt{R'{}^2(t)+Z'{}^2(t)}$
% \begin{equation}
%   \alpha(t) := \sqrt{R'{}^2(t)+Z'{}^2(t)} \label{alpha:def}
% \end{equation}
is the arc-length Jacobian. Any axisymmetric vector field $\bsb{v}$ on $\Gamma$ has the form
\begin{equation}
  \bsb{v}(t,\phi) = v_{\tau}(t)\bsb{\tau}(t,\phi) + v_n(t)\bsb{n}(t,\phi) \qquad\text{with}\quad
  v_{\tau}:=\bsb{v}\cdot\bsb{\tau},\ v_n:=\bsb{v}\cdot\bsb{n}, \qquad 0\le t \le \pi, \, \quad 0\le\phi < 2\pi \label{tn:components}
\end{equation}
which will often be used in the sequel. Then, by the Frenet formulas on $\gamma_{\phi}$,
\begin{equation}
  \pd_t\bsb{x}(t,\phi) = \alpha(t)\bsb{\tau}(t,\phi), \qquad
  \pd_t\bsb{\tau}(t,\phi) = \alpha(t)\kappa(t)\bsb{n}(t,\phi), \qquad
  \pd_t\bsb{n}(t,\phi) = -\alpha(t)\kappa(t)\bsb{\tau}(t,\phi),
  \label{eq:frenet}
\end{equation}
where $\kappa=\kappa(t)$ is the curvature, given by
\begin{equation}
  \alpha^3\kappa = \alpha^2\bsb{n}\cdot\pd_t\bsb{\tau} = - \alpha^2\bsb{\tau}\cdot\pd_t\bsb{n} = Z'R''-R'Z''. \label{kappa:expr}
\end{equation}
The differential area element on $\Gamma$ is $\td S=R(t)\alpha(t) \,\td t \,\td \phi$. For any axisymmetric function $f$ defined on $\Gamma$,
\begin{equation}
  \int_{\Gamma} f\td S
  = \int_0^{2\pi}\int_0^{\pi} f(t)R(t)\alpha(t) \,\td t \,\td \phi = 2\pi\int_{0}^{\pi} f(t)\,R(t)\alpha(t) \,\td t. \label{iG:1D}
\end{equation}

% figure
\begin{figure}[t] \centering
  \includegraphics[width=0.575\textwidth]{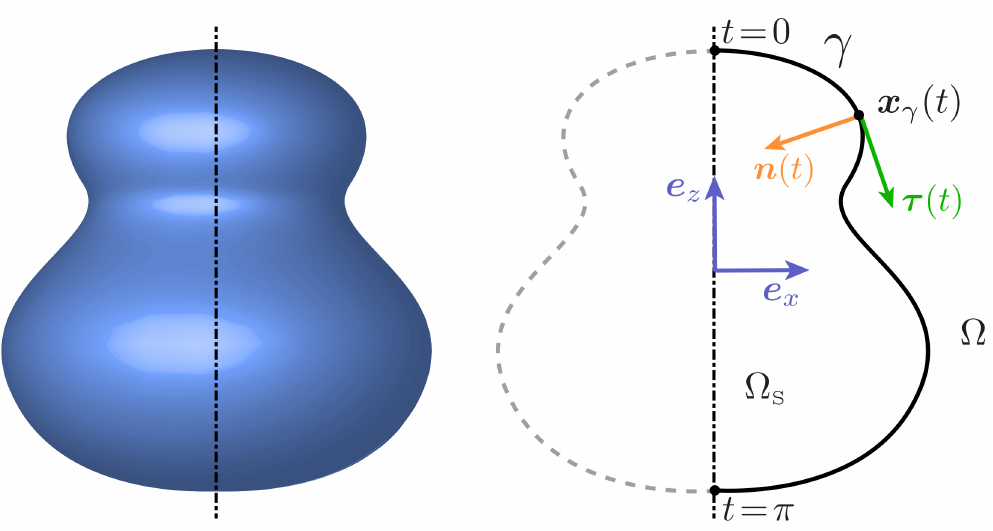}
  \caption{{\em Axisymmetric body of the microswimmer: geometry and notation.}}\label{geom:2D}
\end{figure}

\subsection{PDE of forward problem} Axisymmetric slip velocities are of the form
\begin{equation}
  \bsb{u}^{\sms{S}}(t,\phi)=u^{\sms{S}}(t)\bsb{\tau}(t,\phi) \label{eq:bsbus}
\end{equation}
at any point $\bsb{x}(t,\phi)\in\Gamma$, where $u^{\sms{S}}(t)$ is the slip velocity profile and the unit tangent $\bsb{\tau}$ is defined by~\eqref{tn:axi}. The axisymmetry assumption also implies that the profile $u^{\sms{S}}(t)$ must satisfy
\begin{equation}
u^{\sms{S}}(0) = u^{\sms{S}}(\pi) = 0.  \label{eq:us0}
\end{equation}
to prevent singularities in the flow problem. In the viscous dominant regime, the velocity field $\bsb{u}$ and the pressure field $p$ in the fluid region verify the Stokes PDE system
\begin{equation}
-\mu \nabla^2 \bsb{u} + \nabla p = \bsb{0}, \quad 
\nabla \cdot\bsb{u} = 0, \quad \forall\bsb{x}\in\Omega,  \label{eq:forward}
\end{equation}
where $\mu$ is the dynamic viscosity. In addition, the velocity is prescribed on $\Gamma$ as
\begin{equation}
  \bsb{u} = \bsb{u}^{\sms{D}} := U\bsb{e}_z + u^{\sms{S}}\bsb{\tau} \qquad \text{on }\Gamma, \label{eq:uD}
\end{equation} 
where the axial (along the $z$-axis) translation velocity $U$ of the body is determined by requiring that the force density $\bsb{f}=-p\bsb{n}+\mu(\nabla \bsb{u} + \nabla \tv \bsb{u})\cdot\bsb{n} =( -p\bsb{I}+2\mu D[\bsb{u}])\cdot\bsb{n}$ on $\Gamma$ induced by the flow produces a zero axial net force, i.e.
\begin{equation}
  0 = \intG{\bsb{f}, \bsb{e}_z} = \int_{\Gamma}  (\bsb{f} \cdot \bsb{e}_z) \,\td S
  = 2\pi\int_{0}^{\pi} (\bsb{f}(t) \cdot \bsb{e}_z)R(t)\alpha(t) \,\td t \label{eq:NNF}
\end{equation}
where $\intG{\cdot, \cdot}$ stands for the $L^2(\Gamma)$ duality product. Both the geometry and the slip velocity being (by assumption) axisymmetric, the flow is axisymmetric as well, which prevents rigid-body motions other than axial translations and implies automatic satisfaction of the no-net-torque and remaining no-net-force conditions.

We record the following sign and reciprocity properties:\enlargethispage*{1ex}
\begin{lemma}\label{lemma1}
Let $(\bsb{u}, p, \bsb{f})$ solve problem~\eqref{eq:forward}-\eqref{eq:uD}. Then:
\begin{equation}
  \intG{\bsb{u}^{\sms{D}}, \bsb{f}} = a(\bsb{u},\bsb{u}) \geq 0,
\end{equation}
where the positive bilinear form $a(\cdot,\cdot)$ is defined by~\eqref{eq:weakforward}. In addition, let $(\bsb{u}_1, p_1, \bsb{f}_1)$ define another solution of problem~\eqref{eq:forward}-\eqref{eq:uD}, with prescribed velocity $\bsb{u}_1^{\sms{D}}=U_1 \bsb{e}_z + u_1^{\sms{S}} \bsb{\tau}$. Then:
\begin{equation}
  \intG{\bsb{u}_1^{\sms{D}}, \bsb{f}}=\intG{\bsb{u}^{\sms{D}}, \bsb{f}_1}.
\end{equation}
\end{lemma}
\begin{proof}
Both results stem from the weak versions (given in Section~\ref{prop:Jw:proof}) of equations~\eqref{eq:forward}, \eqref{eq:uD}. For the sign property, write~\eqref{eq:weakforward} for $(\bsb{v},q,\bsb{g})=(\bsb{u},p,\bsb{f})$ and recall that $\nabla\cdot\bsb{u}=0$. For the reciprocity property, write~\eqref{eq:weakforward} for $(\bsb{v},q,\bsb{g}) = (\bsb{u}_1,p_1,\bsb{f}_1)$ and the similar identity obtained by reversing the roles of the two solutions, then subtract the resulting equalities and use the symmetry property $a(\bsb{u}, \bsb{u}_1) = a(\bsb{u}_1,\bsb{u})$.
\end{proof}

We also note, for later reference, that we need to find the {solution} $(\hat{\bsb{u}}, \hat{p}, \hat{\bsb{f}})$ satisfying the rigid body translation problem,
\begin{equation}
  \text{(a) \ } -\mu \nabla^2 \bsb{\hat{u}} + \nabla \hat{p} = \bsb{0}, \quad \nabla \cdot\bsb{\hat{u}} = 0 \quad\text{in }\Omega, \qquad
  \text{(b) \ } \bsb{\hat{u}} = \bsb{e}_z \quad \text{on }\Gamma,
\label{eq:adj0}
\end{equation}
i.e. problem~\eqref{eq:forward}-\eqref{eq:uD} with $(U,u^{\sms{S}}) = (1,0)$ governing the flow created by towing the rigid body $\Omega$ at an unit axial speed. We then have
\begin{equation}
  \bsb{f}[u^{\sms{S}}] = \bsb{f}_0[u^{\sms{S}}] + U(u^{\sms{S}},\Gamma)\hat{\bsb{f}}, \label{eq:f0}
\end{equation}
where $\bsb{f}_0[u^{\sms{S}}]$ denotes the traction associated with the solution of problem~\eqref{eq:forward}-\eqref{eq:uD} with $U=0$. Using the reciprocity property of Lemma~\ref{lemma1}, the no-net-force condition~\eqref{eq:NNF} reads
\begin{equation}
  0 = F_0(\Gamma) U(u^{\sms{S}},\Gamma) + \intG{ \bsb{f}_0[u^{\sms{S}}],\bsb{e}_z }
  = F_0(\Gamma) U(u^{\sms{S}},\Gamma) + \intG{ \hat{f}_{\tau}, u^{\sms{S}} } ,
\label{NNF:exp}
\end{equation}
% \textcolor{red}{$F_0$ is defined on the next page.}
where $\hat{f}_{\tau} := \hat{\bsb{f}}\cdot \bsb{\tau}$ (with subscript $\tau$ used throughout to indicate the tangential projection of the given vector) and the drag coefficient $F_0(\Gamma)$, given by
\begin{equation}
  F_0(\Gamma) = \intG{\hat{\bsb{f}}, \bsb{e}_z}, \label{F0:def}
\end{equation}
is the net force incurred by towing the rigid body at a unit axial speed.

\section{Optimization problems and shape sensitivity analysis}
\label{sc:optimization}

%\subsection{Swimming efficiency}

Our aim is to find an optimal shape $\Gamma$ achieving low-Reynolds locomotion with maximum swimming efficiency. Following \cite{lighthill1952squirming}, the swimming efficiency of the body is defined as
\begin{equation}
J_{\sms{E}}(u^{\sms{S}},\Gamma) : = \frac{J_{\sms{D}}(u^{\sms{S}},\Gamma)}{J_{\sms{W}}(u^{\sms{S}},\Gamma)}  \label{eq:Jeff}
\end{equation}
where
\begin{equation}
  J_{\sms{W}}(u^{\sms{S}},\Gamma) : = \intG{\bsb{f}[u^{\sms{S}}] , \bsb{u}^{\sms{D}} } = \intG{ f_{\tau} , u^{\sms{S}} }, \label{eq:Jw}
\end{equation}
is the power dissipated in the actual motion described by~\eqref{eq:forward}-\eqref{eq:NNF} (with the second equality stemming from equations \eqref{eq:uD} and \eqref{eq:NNF}) and 
\begin{equation}
  J_{\sms{D}}(u^{\sms{S}},\Gamma) := F_0(\Gamma) U^2(u^{\sms{S}}, \Gamma). \label{eq:JD}
\end{equation}
is the power loss caused by towing a rigid body of the same shape at the axial translation velocity $U=U(u^{\sms{S}},\Gamma)$. 
% The powers $J_{\sms{W}}$ and $J_{\sms{D}}$ are respectively given by
% with the second equality stemming from equations \eqref{eq:uD} and \eqref{eq:NNF}, and
% where the drag coefficient $F_0(\Gamma)$ is the net force incurred by towing the rigid body at a unit axial speed. By finding the \emph{adjoint solution} $(\hat{\bsb{u}}, \hat{p}, \hat{\bsb{f}})$ satisfying \eqref{eq:adj0}, the drag coefficient is then the adjoint net force:
% \begin{equation}
%   F_0(\Gamma) = \intG{\hat{\bsb{f}}, \bsb{e}_z}. \label{F0:def}
% \end{equation}
Applying the sign property of Lemma~\ref{lemma1} to~\eqref{eq:Jw} and~\eqref{F0:def}, we deduce $J_{\sms{W}}(u^{\sms{S}},\Gamma) \geq 0$ and $F_0(\Gamma) \geq 0$ for any $\Gamma,\,u^{\sms{S}}$.
% \begin{equation}
%   J_{\sms{W}}(u^{\sms{S}},\Gamma) \geq 0, \qquad F_0(\Gamma) \geq 0 \qquad \text{for any } \Gamma,\,u^{\sms{S}}
% \end{equation}
As a result, the swimming efficiency~\eqref{eq:Jeff} takes the form
\begin{equation}
  J_{\sms{E}}(u^{\sms{S}},\Gamma)
 = \frac{A_{\sms{D}}(u^{\sms{S}},u^{\sms{S}},\Gamma)}{A_{\sms{W}}(u^{\sms{S}},u^{\sms{S}},\Gamma)} \label{eq:Jeff2}
\end{equation}
with the positive bilinear forms $A_{\sms{D}}$ and $A_{\sms{W}}$ given by
\begin{equation}
  A_{\sms{D}}(u^{\sms{S}},w^{\sms{S}},\Gamma) = \intG{ \hat{f}_{\tau} , u^{\sms{S}} }\, \intG{ \hat{f}_{\tau} , w^{\sms{S}} } , \qquad
  A_{\sms{W}}(u^{\sms{S}},w^{\sms{S}},\Gamma) = F_0(\Gamma)\intG{ f_{\tau}[u^{\sms{S}}], w^{\sms{S}}}. \label{AD:AW:def}
\end{equation}

\subsection{Optimization problems}

Our main goal is to solve numerically the optimization problem
\begin{equation}
\max\limits_{u^{\sms{S}}, \Gamma} J_{\sms{E}}(u^{\sms{S}}, \Gamma) \quad \text{subject to} \quad C_{\nu}(\Gamma) := \nu(\Gamma) - \nu_0 = 0,
\label{eq:const_opt}
\end{equation}
where the swimming efficiency is to be maximized with respect to both the applied slip velocity profile $u^{\sms{S}}$ and the body shape $\Gamma$, subject to the reduced volume
\begin{equation}
  \nu(\Gamma) := 6\sqrt{\pi}V/A^{3/2} \label{nu:def}
\end{equation}
having a prescribed value $\nu_0$ (with $V=|\Omega_{\sms{S}}|$ and $A=|\Gamma|$ the volume and surface area of the microswimmer). We note that the reduced volume is a function of swimmer's shape, not size. That is, scaling the swimmer uniformly in all directions does not change the reduced volume. Additionally, we have $\nu(\Gamma)\leq1$ for all shapes, and  $\nu(\Gamma)=1$ if and only if $\Gamma$ is a sphere.

To solve problem~\eqref{eq:const_opt}, we take advantage of an available solution method for the partial maximization of $J_{\sms{E}}(u^{\sms{S}}, \Gamma)$ with fixed shape, see Section~\ref{sec:slipopt}. We thus define
\begin{equation}
  E(\Gamma) := \max_{u^{\sms{S}}} J_{\sms{E}}(u^{\sms{S}}, \Gamma) \label{E:def}
\end{equation}
and note that $E(\Gamma)$ is a shape functional since its value is entirely determined by $\Gamma$. We then recast the joint optimization problem~\eqref{eq:const_opt} as the constrained shape optimization problem
\begin{equation}
\max\limits_{\Gamma} E(\Gamma) \quad \text{subject to} \quad C_{\nu}(\Gamma) = 0.
 \label{eq:optJM}
\end{equation}
The equality constraint will be handled using the augmented Lagrangian method (ALM) described in Section \ref{sc:NumSch}, the constrained problem~\eqref{eq:optJM} thus being treated as a sequence of unconstrained problems.

We will also consider, for comparison purposes, the constrained shape optimization problem
\begin{equation}\label{eq:opt_drag}
  \min\limits_{\Gamma} J_{\sms{drag}} (\Gamma) \quad \text{subject to} \quad \nu(\Gamma) - \nu_0 = 0.
\end{equation}
for the normalized drag force $J_{\sms{drag}} (\Gamma)$, which is defined by $J_{\sms{drag}}(\Gamma) := F_0(\Gamma)/6\pi \mu r$
% \begin{equation}\label{eq:Jdrag}
%     J_{\sms{drag}}(\Gamma) := \frac{F_0(\Gamma)}{6\pi \mu r}
% \end{equation}
for a (passive) rigid body in Stokes flow ($r=\sqrt[3]{3V/(4\pi)}$ being the radius of the ball having the same volume $V$ as $\Omega_{\sms{S}}$). It is worth noting that $J_{\sms{drag}}(\Gamma)$ is dimensionless and does not depend on the size of $\Omega_{\sms{S}}$.

\subsection{Partial optimization with fixed shape}
\label{sec:slipopt}

Here we present an improved version of the method proposed in~\cite{guo2021slip} for solving the slip optimization problem~\eqref{E:def} with a fixed shape $\Gamma$. This step exploits the fact that $J_{\sms{E}}(u^{\sms{S}},\Gamma)$ is, for given $\Gamma$, expressed by~\eqref{eq:Jeff2} as a Rayleigh quotient, so that its maximum value is equal to the largest (positive) eigenvalue of the symmetric generalized eigenvalue problem
\begin{equation}
  \text{Find }z^{\sms{S}}\in H^{1/2}(\Gamma), \lambda\in\mathbb{R} \quad \text{such that} \qquad
  A_{\sms{D}}(z^{\sms{S}},w^{\sms{S}},\Gamma) - \lambda A_{\sms{W}}(z^{\sms{S}},w^{\sms{S}},\Gamma) = 0 \qquad \text{for all }w^{\sms{S}}\in H^{1/2}(\Gamma) \label{eigv:JE}
\end{equation}
%\textcolor{red}{$\mathcal{U}$ is undefined.}
the corresponding eigenfunction $z^{\sms{S}}$ defining a slip velocity profile producing the optimal efficiency. Problem~\eqref{eigv:JE} turns out to be easy to solve: since the bilinear form $A_{\sms{D}}$ is of rank one (while $A_{\sms{W}}$ is positive definite), there is only one nonzero eigenvalue $\lambda>0$, whose multiplicity is 1.

Let the fields $(\tilde{\bsb{u}},\tilde{p})$ solve the axisymmetric Stokes flow problem
\begin{equation}
  \text{(a) \ }-\mu \nabla^2 \tilde{\bsb{u}} + \nabla\tilde{p} = \bsb{0}, \quad
  \nabla \cdot\tilde{\bsb{u}} = 0, \quad \text{in }\Omega, \qquad
  \text{(b) \ }\tilde{\bsb{f}}\cdot\bsb{\tau}=\hat{\bsb{f}}\cdot\bsb{\tau}, \qquad
  \text{(c) \ }\tilde{\bsb{u}}\cdot\bsb{n}=0 \quad \text{on }\Gamma, \label{eq:optslip}
\end{equation}
where $\tilde{\bsb{f}}$ is the traction vector associated with $(\tilde{\bsb{u}},\tilde{p})$. Using the above solution, the slip velocity defined by
\begin{equation}
  z^{\sms{S}} := \tilde{\bsb{u}}\cdot\bsb{\tau}, \label{zS:def}
\end{equation}
% and note for later use that the traction vectors $\bsb{f}[z^{\sms{S}}]$, $\hat{\bsb{f}}$ and the swimming velocity $U[z^{\sms{S}}]$ verify
% \begin{equation}
%   f_{\tau}[z^{\sms{S}}] = \tilde{f}_{\tau} + U[z^{\sms{S}}] \hat{f}_{\tau} = \big( 1+U[z^{\sms{S}}] \big) \hat{f}_{\tau} %\label{fs:fsh}
% \end{equation}
% as a result of taking the tangential projection of~\eqref{eq:f0} and using the boundary condition~(\ref{eq:optslip}b). 
then yields the maximal swimming efficiency for a fixed shape $\Gamma$:
\begin{proposition}\label{zS:opt}
The slip velocity profile $z^{\sms{S}}$ defined for a given shape $\Gamma$ by~\eqref{zS:def} in terms of the solution of problem~\eqref{eq:optslip} solves the partial maximization problem~\eqref{E:def}: we have
\begin{equation}
  z^{\sms{S}} = \argmax_{u^{\sms{S}}} J_{\sms{E}}(u^{\sms{S}},\Gamma), \qquad
  E(\Gamma) = -\frac{U(z^{\sms{S}},\Gamma)}{1+U(z^{\sms{S}},\Gamma)} \geq 0, \label{EG:expr}
\end{equation}
\end{proposition}
\begin{proof}
Let $z^{\sms{S}}$ be given by~\eqref{zS:def}. Using~\eqref{NNF:exp} and~\eqref{fs:fsh} in~\eqref{AD:AW:def}, we have
\begin{alignat}{2}
  A_{\sms{D}}(z^{\sms{S}},w^{\sms{S}},\Gamma)
  &= \intG{ \hat{f}_{\tau}, z^{\sms{S}}}\,\intG{ \hat{f}_{\tau}, w^{\sms{S}}}
  &&= -F_0(\Gamma) U(z^{\sms{S}},\Gamma)\,\intG{ \hat{f}_{\tau}, w^{\sms{S}}} \\
  A_{\sms{W}}(z^{\sms{S}},w^{\sms{S}},\Gamma)
  &= F_0(\Gamma)\intG{ f_{\tau}[z^{\sms{S}}], w^{\sms{S}}}
  &&= F_0(\Gamma) \big( 1+U(z^{\sms{S}},\Gamma) \big) \intG{ \hat{f}_{\tau}, w^{\sms{S}}}.
\end{alignat}
The eigenvalue equation~\eqref{eigv:JE} is therefore verified by $z^{\sms{S}}$ and $\lambda=E(\Gamma)$ given by~\eqref{EG:expr}. Since problem~\eqref{eigv:JE} has only one nonzero eigenvalue, the pair $(z^{\sms{S}},\lambda)$ provides the maximal swimming efficiency for given $\Gamma$.
\end{proof}
The treatment given in~\cite{guo2021slip} needs the solution of one flow problem per basis function of the finite-dimensional approximation of $u^{\sms{S}}$ (see Section~\ref{sc:finite}), while the present version replaces this task with solving only~\eqref{eq:optslip} (the solution of the adjoint problem~\eqref{eq:adj0} being needed in both versions).

\subsection{Shape sensitivity analysis}

In this section, we collect available shape derivative concepts that fit the present needs. Rigorous expositions of shape sensitivity theory are available in~\cite[Chap. 5]{henrot2018} and other monographs.
Let $\Omega_{\sms{A}}$ denote a fixed domain chosen so that $\Omega\Subset \Omega_{\sms{A}}$ always holds for the shape optimization problem of interest. 
Shape changes are described with the help of transformation velocity fields, i.e., vector fields $\bsb{\theta} : \Omega_{\sms{A}}\to \mathbb{R}^3$ such that $\bsb{\theta}=\bsb{0}$ in a neighborhood of $\pd\Omega_{\sms{A}}$; we then extend $\bsb{\theta}$ by zero in $\mathbb{R}^3 \!\setminus\! \overline{\Omega_{\sms{A}}}$. Shape perturbations of the body boundary can, in the present context, be mathematically described using a pseudo-time $\eta$ and a geometrical transform of the form
\begin{equation}
\mathbb{R}^3\ni\bsb{x}(t,\phi)   \mapsto \bsb{x}_{\eta}(t,\phi) = \bsb{x}(t,\phi) + \eta \bsb{\theta}(t,\phi), \qquad 0\le t \le\pi,\quad 0\le\phi < 2\pi.
\label{eq:xeta}
\end{equation}
They allow to define a parametrized family of domains $\Omega_\eta (\bsb{\theta}) : = (\bsb{I}+\eta\bsb{\theta})(\Omega)$ for any given ``initial" domain $\Omega$. The affine format \eqref{eq:xeta} is sufficient for defining first-order derivatives at $\eta=0$. In this work, perturbations $\Gamma_{\eta}(\bsb{\theta})$ of an axisymmetric surface $\Gamma$ are assumed to be defined by
\begin{equation}
  \Gamma_{\eta}(\bsb{\theta})\ni \bsb{x}_{\eta}(t,\phi)
  = \bsb{x}(t,\phi) + \eta\bsb{\theta}(t,\phi), \quad 0\le t \le \pi, \quad 0\le\phi < 2\pi \label{eq:perturbgamma}
\end{equation}
in terms of axisymmetry-preserving transformation velocity fields $\bsb{\theta}$ having the form
\begin{equation}\label{eq:theta}
\bsb{\theta}(t,\phi)=\theta_{\tau}(t)\bsb{\tau}(t,\phi) + \theta_n(t)\bsb{n}(t,\phi), \qquad 0\le t \le\pi,\quad 0\le\phi < 2\pi
\end{equation}
on $\Gamma$, with $\bsb{\tau}$ and $\bsb{n}$ given by~\eqref{tn:axi} and the components $\theta_{\tau},\theta_n$ satisfying
\begin{equation}
\text{(a)}\quad \theta_{\tau}(0)=\theta_{\tau}(\pi)=0,\quad \text{(b)} \quad \theta'_n(0)=\theta'_n(\pi)=0.  \label{eq:thetacondition}
\end{equation}
The requirement (\ref{eq:thetacondition}a) ensures that $\Gamma_\eta$ remains closed (by precluding ``tearing" at the poles), and the smoothness of $\Gamma_{\eta}(\bsb{\theta})$ at the poles is maintained by (\ref{eq:thetacondition}b). Perturbed domains $\Omega_\eta (\bsb{\theta}) = (\bsb{I}+\eta\bsb{\theta})(\Omega)$ with boundary $\Gamma_{\eta}(\bsb{\theta})$ can then be defined using arbitrary extensions of $\bsb{\theta}$ to $\Omega$.

All derivatives are implicitly taken at some given configuration $\Omega$, i.e., at initial ``time" $\eta=0$. The ``initial" material derivative $\overstar{\bsb{a}}$ of some (scalar or tensor-valued) field variable $\bsb{a}(\bsb{x},\eta)$ is defined as
\begin{equation}
\overstar{\bsb{a}}(\bsb{x}) = \lim_{\eta\to 0} \frac{1}{\eta} \Big( \bsb{a}(\bsb{x}_{\eta},\eta) - \bsb{a}(\bsb{x},0)\Big)
\quad \bsb{x}\in \Omega,
\end{equation}
and the material derivative of gradients and divergences of tensor fields are given by
\begin{equation}
\text{(a)} \quad (\nabla \bsb{a})^{\scriptsize\star} = \nabla \overstar{\bsb{a}}-\nabla \bsb{a} \cdot \nabla \bsb{\theta},
\quad
\text{(b)} \quad (\nabla\cdot \bsb{a})^{\scriptsize\star} = \nabla\cdot \overstar{\bsb{a}}-\nabla \bsb{a} \!:\! \nabla \bsb{\theta}.
\label{dd:grad}
\end{equation}
Likewise, the first-order ``initial'' derivative $J'$ of a shape functional $J$ is defined as
\begin{equation}
  J'(\Omega;\bsb{\theta}) = \lim_{\eta\to0} \frac{1}{\eta}\Big( J(\Omega_{\eta}(\bsb{\theta}))-J(\Omega) \Big), \label{dd:J:def}
\end{equation}
and its practical evaluation relies on the fact that the derivatives of generic integrals
\begin{equation}
  \text{(a)} \quad I_{\sms{V}}(\eta) = \int_{\Omega_{\eta}(\bsb{\theta})} F(\cdot,\eta) \td V, \qquad
  \text{(b)} \quad I_{\sms{S}}(\eta) = \int_{\Gamma_{\eta}(\bsb{\theta})} F(\cdot,\eta) \td S, \label{IVS:def}
\end{equation}
on variable domains or surfaces are given by the material differentiation identities
\begin{equation}
  \text{(a) \ } \frac{\td I_{\sms{V}}}{\td\eta}\Big|_{\eta=0} = \int_{\Omega} \big[ \overstar{F} + F(\cdot,0)\,\nabla\cdot\bsb{\theta} \big] \td V, \qquad
  \text{(b) \ } \frac{\td I_{\sms{S}}}{\td\eta}\Big|_{\eta=0} = \int_{\Gamma} \big[ \overstar{F} + F(\cdot,0)\,\divs\bsb{\theta} \big] \td S. \label{dd:IVS}
\end{equation}
with the surface divergence $\divs \bsb{\theta}$ in~(\ref{dd:IVS}b) given by~\eqref{divS:def} in Appendix \ref{apd:proofprop}.

The \textit{structure theorem for shape derivatives} (see e.g. \cite[Sec.~5.9]{henrot2018}) states that the shape derivative of any shape functional can be expressed as a linear functional in the normal transformation velocity $\theta_n$. This result conforms to intuitive geometrical facts: (i) the shape of $\Omega_{\eta}(\bsb{\theta})$ is determined by that of $\Gamma_{\eta}(\bsb{\theta})$, and (ii) tangential components of $\bsb{\theta}$ leave $\Omega$ unchanged at leading order $O(\eta)$. Here we provide, as an example and for later reference, the derivative of the reduced volume~\eqref{nu:def}, which is a shape functional:
\begin{proposition}\label{prop:geom}
The shape derivative of the reduced volume $\nu(\Gamma)$ defined by~\eqref{nu:def} is given by
\begin{equation}
  \frac{\nu'(\Gamma;{\bsb{\theta}})}{\nu(\Gamma)}
 = \frac{V'(\Gamma;{\bsb{\theta}})}{V} - \frac{3A'(\Gamma;{\bsb{\theta}})}{2A}. \label{dd:nu}
\end{equation}
with the shape derivatives of the volume $V=|\Omega_{\sms{S}}|$ and area $A=|\Gamma|$ given (as linear functionals on $\theta_n$) by
\begin{equation}
  V'(\Gamma;{\bsb{\theta}}) = -\int_{\Gamma}\theta_n \,\td S = -2\pi\!\int_0^\pi \theta_n R\alpha \,\td t, \qquad
  A'(\Gamma;{\bsb{\theta}}) = \int_{\Gamma}  \divs {\bsb{\theta}} \,\td S = 2\pi\int_0^\pi \big(Z' - \kappa R\alpha \big)\theta_n  \,\td t. \label{eq:Sprime}
\end{equation}
\end{proposition}
For $A'$, we have used~\eqref{divS:def} and the fact that $\int_0^{\pi} (R\theta_{\tau})' \,\td t=0$ due to (\ref{eq:thetacondition}a).

\subsection{Shape sensitivities of swimming efficiency and normalized drag force}
\label{sec:shape_sens}

We begin by expressing the shape derivative of the swimming efficiency in terms of shape sensitivities of the power loss functionals:

\begin{lemma}\label{dd:lambdamax}
The shape derivative of $E(\Gamma)$ is given by
\begin{equation}
  E'(\Gamma;\bsb{\theta})
 = \frac{J'_{\sms{D}}(z^{\sms{S}},\Gamma;\bsb{\theta}) - E(\Gamma)J'_{\sms{W}}(z^{\sms{S}},\Gamma;\bsb{\theta})}{J_{\sms{W}}(z^{\sms{S}},\Gamma)}
\end{equation}
where $E(\Gamma)$ and the shape derivatives $J'_{\sms{W}}$ and $J'_{\sms{D}}=\big[ F_0U^2 \big]'$ are taken with the slip velocity $z^{\sms{S}}$ given in \eqref{zS:def} (kept fixed to its optimum value at current $\Gamma$).
\end{lemma}
\begin{proof}
The maximal efficiency $E(\Gamma)$ and associated optimal slip velocity $z^{\sms{S}}$ at any given $\Gamma$ are related through equation~\eqref{eigv:JE}, i.e.:
\begin{equation}
  A_{\sms{D}}(z^{\sms{S}},w^{\sms{S}},\Gamma) - E(\Gamma) A_{\sms{W}}(z^{\sms{S}},w^{\sms{S}},\Gamma) = 0. \label{eigv:eq}
\end{equation}
for any shape perturbation about the current shape $\Gamma$. The shape derivative at $\Gamma$ of \eqref{eigv:eq} thus yields
\begin{equation}
  A_{\sms{D}}(\overstar{z}{}^{\sms{S}},w^{\sms{S}},\Gamma) - E(\Gamma) A_{\sms{W}}(\overstar{z}{}^{\sms{S}},w^{\sms{S}},\Gamma)
   + A'_{\sms{D}}(z^{\sms{S}},w^{\sms{S}},\Gamma;\bsb{\theta}) - E(\Gamma) A'_{\sms{W}}(z^{\sms{S}},w^{\sms{S}},\Gamma;\bsb{\theta})
   - E'(\Gamma;\bsb{\theta}) A_{\sms{W}}(z^{\sms{S}},w^{\sms{S}},\Gamma) = 0
\end{equation}
where $A'_{\sms{D}}$ and $A'_{\sms{W}}$ are shape derivatives taken with $z^{\sms{S}}$ fixed (i.e. $\overstar{z}{}^{\sms{S}}=0$), while as usual $\overstar{w}{}^{\sms{S}}=0$ may be assumed for the test functions in this derivation. Now, setting $w^{\sms{S}}=z^{\sms{S}}$ in the above equation and observing that the first two terms cancel due to~\eqref{eigv:eq} and the bilinear forms $A_{\sms{D}},A_{\sms{W}}$ being symmetric, we obtain
\begin{equation}
  E'(\Gamma;\bsb{\theta})
 = \frac{A'_{\sms{D}}(z^{\sms{S}},z^{\sms{S}},\Gamma;\bsb{\theta}) - E(\Gamma) A'_{\sms{W}}(z^{\sms{S}},z^{\sms{S}},\Gamma;\bsb{\theta})}{A_{\sms{W}}(z^{\sms{S}},z^{\sms{S}},\Gamma)}.
\end{equation}
The claimed formula finally results from recalling the definitions~\eqref{AD:AW:def} of $A_{\sms{D}},A_{\sms{W}}$, which imply $A_{\sms{D}}(z^{\sms{S}},w^{\sms{S}},\Gamma)=F_0(\Gamma) J_{\sms{D}}(z^{\sms{S}},\Gamma)$ and $A_{\sms{W}}(z^{\sms{S}},w^{\sms{S}},\Gamma)=F_0(\Gamma) J_{\sms{W}}(z^{\sms{S}},\Gamma)$.
\end{proof}

The next step then consists in deriving formulas for the shape sensitivities of the functionals involved in Lemma~\ref{dd:lambdamax}. The latter depend on $\Gamma$ implicitly through the forward or adjoint solution. In particular, applying the material differentiation formula~(\ref{dd:IVS}b) to~\eqref{eq:Jw} and~\eqref{F0:def}, we have
\begin{align}
  J_{\sms{W}}'(\Gamma;\bsb{\theta})
 &= \intG{ \overstar{\bsb{f}} , u^{\sms{S}}\bsb{\tau} } + \intG{ \overstar{\bsb{f}} , \bsb{e}_z }U
 + \intG{ \bsb{f} , u^{\sms{S}}\overstar{\bsb{\tau}} }
 + \intG{ \bsb{f} , (u^{\sms{S}}\bsb{\tau} + U\bsb{e}_z)\,\divs \bsb{\theta} } \label{eq:JWprime1} \\
  F'_0(\Gamma;\bsb{\theta})
 &= \intG{\overstar{\hat{\bsb{f}}}, \bsb{e}_z} + \intG{\hat{\bsb{f}}, \bsb{e}_z\divs \bsb{\theta}}. \label{eq:F0prime}
\end{align}
in terms of the material derivatives $\overstar{\bsb{f}} , \overstar{U}$ and $\overstar{\hat{\bsb{f}}}$ of solution components and $\overstar{\bsb{\tau}}$ of the unit tangent, and having used the no-net-force identity $\intG{ \bsb{f}, \bsb{e}_z } \overstar{U}= 0$ for~\eqref{eq:JWprime1}. However, as in many other similar situations, expressions of $J_{\sms{W}}'(\Gamma;\bsb{\theta})$, $F'_0(\Gamma;\bsb{\theta})$ and $U'(\Gamma;\bsb{\theta})$ free of solution derivatives can be obtained from combinations of the weak forms of the forward and derivative problems written with suitably chosen test functions. In addition, the resulting expressions are recast, using curvilinear coordinates, as boundary integrals. This somewhat lengthy process, expounded in Section~\ref{prop:Jw:proof}, yields the following material derivative formulas, which are free of any solution material derivative and have a form suitable for a direct implementation using the output of a BIE solver:\enlargethispage*{1ex}
\begin{proposition}\label{prop:Jw}
Consider a shape perturbation with transformation velocity $\bsb{\theta}$, and assume the slip velocity $u^{\sms{S}}$ to be convected (i.e. $\overstar{u}{}^{\sms{S}}=0$). The derivatives of $J_{\sms{W}}(u^{\sms{S}},\Gamma)$, $U(u^{\sms{S}},\Gamma)$ and $F_0(\Gamma)$ are then given by
%\begin{subequations}
\begin{align}
 J'_{\sms{W}}(u^{\sms{S}},\Gamma;\bsb{\theta} )
 &= 2\pi \!\! \int_0^\pi \Big\{ \Big[ 4\mu \Big( \frac{R'u^{\sms{S}}}{R\alpha} \Big)^2
  -\frac{1}{\mu} f_{\tau}^2
  + \frac{1}{\mu}(f_n+p)p + 2\kappa u^{\sms{S}} f_{\tau} \Big] \alpha\theta_n
  - 2f_{\tau}(u^{\sms{S}} )'\theta_{\tau} + 2u^{\sms{S}} f_n\theta_n ' \Big\} R\td t \hspace*{-2em} \label{eq:JWprime} \\
 U'(u^{\sms{S}},\Gamma;\bsb{\theta} )
 &= - \frac{2\pi}{F_0}  \int_{0}^{\pi} \Big\{ \Big[ \kappa u^{\sms{S}} \hat{f}_{\tau} - \frac{1}{\mu} f_{\tau}\hat{f}_{\tau} + \frac{1}{2\mu}(f_n+p)\hat{p} \Big] \alpha\theta_n
  - \hat{f}_{\tau}(u^{\sms{S}} )'\theta_{\tau} - u^{\sms{S}} \hat{p}\theta_n ' \Big\} R\td t \label{eq:Uprime} \\
  F_0'(\Gamma;\bsb{\theta} )
 &= -\frac{2\pi}{\mu} \int_{0}^{\pi} \hat{f}_{\tau}^2\theta_n R\alpha \,\td t .
 \retainlabel{eq:Fprime}
\end{align}
with $f_{\tau}:=\bsb{f}\cdot \bsb{\tau}$ and $f_n:=\bsb{f}\cdot \bsb{n}$. The shape derivative of the normalized drag force is then given by
\begin{equation}
 %    J_{\text{drag}}'(\Gamma;\bsb{\theta})
 % = \frac{1}{6\pi\mu} \left(\frac{4\pi}{3}\right)^{1/3} \left(\frac{F_0'}{V^{1/3}} - \frac{F_0 V'}{3V^{4/3}}\right), \label{eq:JDprime}
    J_{\text{drag}}'(\Gamma;\bsb{\theta})
 = J_{\text{drag}}(\Gamma) \left(\frac{F_0'}{F_0} - \frac{V'}{3V} \right), \label{eq:JDprime}
\end{equation}
%\end{subequations}
with $V'$ given by Proposition~\ref{prop:geom}.
\end{proposition}

Formulas~\eqref{eq:JWprime} and~\eqref{eq:Uprime} are valid for any slip velocity profile $u^{\sms{S}}$ that is convected by $\bsb{\theta}$. Moreover, (\ref{eq:JWprime}-d) are all insensitive to a perturbation of the (forward or adjoint) pressure field by a constant pressure difference $\Delta p$. For example, replacing $f_n$ and $p$ by $f_n-\Delta p$ and $p+\Delta p$ brings to \eqref{eq:JWprime} the additional term
\begin{equation}
-4\pi \Delta p \int_0^{\pi} \big[ Ru^{\sms{S}} \theta_n' + (R u ^{\sms{S}})'\theta_n \big] \td t = -4\pi\Delta p \int_0^{\pi} (Ru^{\sms{S}} \theta_n)' \,\td t= 0 \quad \text{using \eqref{eq:us0}}.
\end{equation}
Applying~\eqref{eq:JWprime}, \eqref{eq:Uprime} with $u^{\sms{S}}=z^{\sms{S}}$ to Lemma~\ref{dd:lambdamax} defines a computationally tractable evaluation method for the shape derivative of the swimming efficiency. We note however that the resulting formula for $E'(\Gamma;\bsb{\theta})$ appears to involve the tangential transformation velocity $\theta_{\tau}$, in apparent violation of the structure theorem for shape derivatives. To resolve this contradiction, we now exploit additional properties satisfied only by the optimal slip velocity $z^{\sms{S}}$. Indeed, the definition~\eqref{zS:def} of $z^{\sms{S}}$ and the boundary conditions of problems~\eqref{eq:forward} and~\eqref{eq:adj0} together imply that $(\tilde{\bsb{u}},\tilde{p})$ solve problem~\eqref{eq:forward} with $(u^{\sms{S}},U)=(z^{\sms{S}},0)$, so that $\bsb{f}[z^{\sms{S}}] = \tilde{\bsb{f}} + U(z^{\sms{S}},\Gamma)\hat{\bsb{f}}$. Hence, using again the boundary condition~(\ref{eq:optslip}b), the traction components of the forward solution are found to verify\begin{equation}
  f_{\tau}[z^{\sms{S}}] = \tilde{f}_{\tau} + U(z^{\sms{S}},\Gamma) \hat{f}_{\tau} = \big( 1+U(z^{\sms{S}},\Gamma) \big) \hat{f}_{\tau}, \qquad
  f_n[z^{\sms{S}}] = \tilde{f}_n-U(z^{\sms{S}},\Gamma)\hat{p}, \qquad p[z^{\sms{S}}] = \tilde{p}+U(z^{\sms{S}},\Gamma)\hat{p} \label{fs:fsh}.
\end{equation}
This results in the following final expression of the shape derivative of $E(\Gamma)$, whose form is now consistent with the structure theorem:
\begin{proposition}\label{prop:dE}
Let $z^{\sms{S}}$ be the optimal slip velocity~\eqref{zS:def} for given $\Gamma$, and consider a shape perturbation of $\Gamma$ with transformation velocity field $\bsb{\theta}$. The shape derivative of the optimal efficiency $E(\Gamma)$ is given by
% \begin{multline}
%     E'(\Gamma;\bsb{\theta})
%  = \frac{1}{J_{\sms{W}}(z^{\sms{S}},\Gamma)} \frac{2\pi U}{1\!+\! U} \int_0^{\pi} \Big\{ 2z^{\sms{S}} \big( f_n + (1\!+\! U)\hat{p} \big) \theta_n ' \\
%   + \Big[ 4\mu \Big( \frac{R'z^{\sms{S}}}{R\alpha} \Big)^2 + \frac{1}{\mu}(f_n+p)\big( p - (1\!+\! U)\hat{p} \big)
%   + \frac{1\!+\! U}{\mu} \hat{f}{}^2_{\tau} \Big] \alpha\theta_n \Big\} R\td t
% \end{multline}
\begin{align} \label{eq:ddE}
    E'(\Gamma;\bsb{\theta})
 &= -\frac{2\pi}{F_0(1\!+\! U)^2}
% = \frac{1}{J_{\sms{W}}(z^{\sms{S}},\Gamma)} \frac{2\pi U}{1\!+\! U} 
  \int_0^{\pi} \Big\{ 2z^{\sms{S}} \big( \tilde{f}_n + \hat{p} \big) \theta_n ' \\ & \hspace*{8em}
  + \Big[ 4\mu \Big( \frac{R'z^{\sms{S}}}{R\alpha} \Big)^2 + \frac{1}{\mu}\big( \tilde{f}_n + \tilde{p} \big)\,\big( \tilde{p}-\hat{p} \big)
  + \frac{1\!+\! U}{\mu} \hat{f}{}^2_{\tau} \Big] \alpha\theta_n \Big\} R\td t \notag
\end{align}
where $(\hat{\bsb{u}},\hat{\bsb{f}},\hat{p})$ and $(\tilde{\bsb{u}},\tilde{\bsb{f}},\tilde{p})$ respectively solve problems~\eqref{eq:adj0} and~\eqref{eq:optslip}, $z^{\sms{S}}=\tilde{\bsb{u}}\cdot\bsb{\tau}$ and $U=U(z^{\sms{S}},\Gamma)$ is given by~\eqref{NNF:exp}. In particular, $E'(\Gamma;\bsb{\theta})$ is a linear functional on $\theta_n$.
% $(\bsb{f},p,U)$ solve the forward problem with $u^{\sms{S}}=z^{\sms{S}}$. In particular, $E'(\Gamma;\bsb{\theta})$ is a linear functional on $\theta_n$.
\end{proposition}
\begin{proof}
We evaluate formulas~(\ref{eq:JWprime}-c) for the optimal slip velocity $z^{\sms{S}}$, which allows to express $f_{\tau}$, $f_n$  and $p$ in terms of $\tilde{f}_n,\tilde{p}$ and $\hat{f},\hat{p}$ using~\ref{fs:fsh}.
% , which satisfies the additional property~\ref{fs:fsh}. Accordingly, every occurrence of $f_{\tau}$ in~(\ref{eq:JWprime}-c) is replaced by $(1\!+\! U)\hat{f}_{\tau}$. 
We then use the result to compute $E'(\Gamma;\bsb{\theta})$ given by Lemma~\eqref{dd:lambdamax} and recall that $J_{\sms{W}}(z^{\sms{S}},\Gamma)=E(\Gamma)J_{\sms{D}}(z^{\sms{S}},\Gamma)=-F_0U(1\!+\!U)$ (see~\eqref{eq:JD} and~\eqref{EG:expr}. This yields the claimed expression of $E'(\Gamma;\bsb{\theta})$.
\end{proof}

\subsection{Proof of Proposition~\ref{prop:Jw}}
\label{prop:Jw:proof}

This proof is divided into five main steps.

\proofstep{1. Forward and adjoint problems in weak form}
The results of Proposition~\ref{prop:Jw} rely on identities found by recasting the forward problem~\eqref{eq:forward}-\eqref{eq:uD} in mixed weak form (e.g. \cite{brez:fort:91}, Chap. 6): find $(\bsb{u}, p, \bsb{f})\in \bsb{\mathcal{V}} \times  \mathcal{P} \times \bsb{\mathcal{F}}$ such that
\begin{equation}
\label{eq:weakforward}
  a(\bsb{u}, \bsb{v}) - b(\bsb{v}, p) - b(\bsb{u}, q) - \intG{ \bsb{f}, \bsb{v} }
  + \intG{ \bsb{g}, u^{\sms{S}} \bsb{\tau}+U\bsb{e}_{z} } - \intG{ \bsb{g}, \bsb{u} } = 0 \qquad
  \forall (\bsb{v}, q,\bsb{g}) \in \bsb{\mathcal{V}}\times \mathcal{P}\times \mathcal{F}
\end{equation}
where the bilinear forms $a$ and $b$ are defined by
\begin{equation}
a(\bsb{u}, \bsb{v}) = \int_{\Omega} 2\mu ( \bsb{D}[\bsb{u}]\!:\!\bsb{D}[\bsb{v}]) \, \td V,  \qquad
b(\bsb{v}, q) = \int_{\Omega} q\, (\nabla \cdot \bsb{v}) \, \td V,
\end{equation}
The weak problem~\eqref{eq:weakforward} is well-posed if $\bsb{\mathcal{F}}=H^{-1/2}(\Gamma;\mathbb{R}^3)$, $\mathcal{P}=L^2(\Omega)$ and $\bsb{\mathcal{V}}$ is a weighted version of $H^1(\Omega;\mathbb{R}^3)$. Supplementing problem~\eqref{eq:weakforward} with the no-net-force condition~\eqref{eq:NNF} determines $U$ given $u^{\sms{S}}$. The unknown $\bsb{f}$, acting as the Lagrange multiplier associated with the Dirichlet BC, is in fact the force density on $\Gamma$ given by $\bsb{f}=\bsb{\sigma}[\bsb{u},p] \cdot \bsb{n}$, where $\bsb{\sigma}[\bsb{u},p] = -p\bsb{I}+2\mu\bsb{D}[\bsb{u}]$ is the stress tensor. Similarly, the adjoint problem~\eqref{eq:adj0} in weak form is: find $(\hat{\bsb{u}},\hat{p},\hat{\bsb{f}})\in \bsb{\mathcal{V}}\times \mathcal{P}\times\bsb{\mathcal{F}}$ such that
\begin{equation}
\label{eq:weakadjoint}
  a(\hat{\bsb{u}}, \bsb{v}) - b(\bsb{v}, \hat{p}) - b(\hat{\bsb{u}}, q) - \intG{ \hat{\bsb{f}}, \bsb{v} }
  + \intG{ \bsb{g}, \bsb{e}_{z} } - \intG{ \bsb{g}, \bsb{u} } = 0 \qquad
  \forall (\bsb{v}, q,\bsb{g}) \in \bsb{\mathcal{V}}\times \mathcal{P}\times \mathcal{F}
\end{equation}

\proofstep{2. Weak formulation for material derivatives of the forward solution.} The governing weak formulation for the shape derivative $(\overstar{\bsb{u}},\overstar{\bsb{f}}, \overstar{p},\overstar{U})$ of the solution $(\bsb{u},\bsb{f},p, U)$ of the forward problem \eqref{eq:weakforward}-\eqref{eq:NNF} is
\begin{align}
\lefteqn{
  a(\overstar{\bsb{u}}, \bsb{v}) - b(\bsb{v}, \overstar{p}) - b(\overstar{\bsb{u}}, q) - \intG{ \overstar{\bsb{f}}, \bsb{v} }
 + \intG{ u^{\sms{S}}\overstar{\bsb{\tau}} , \bsb{g} } + \intG{ \bsb{e}_z , \bsb{g} }\,\overstar{U}
- \intG{ \overstar{\bsb{u}} , \bsb{g} } } & \\
 &\qquad=  - \intO{ \bsb{E}\big( (\bsb{u},p), (\bsb{v},q)\big) , \nabla \tv \bsb{\theta}}  + \intG{ \bsb{f}, \bsb{v} \,\divs \bsb{\theta} } \qquad
  \forall (\bsb{v}, q,\bsb{g}) \in \bsb{\mathcal{V}}\times \mathcal{P}\times \mathcal{F}, \label{eq:weakderivative}
\end{align}
with the symmetric in $\big( (\bsb{u},p), (\bsb{v},q)\big)$ tensor-valued function $\bsb{E}$ defined by
\begin{equation}
\bsb{E}\big( (\bsb{u},p), (\bsb{v},q)\big) = \big( 2\mu \bsb{D}[\bsb{u}]\!:\!\bsb{D}[\bsb{v}] - p (\nabla\cdot\bsb{v}) - q (\nabla\cdot\bsb{u}) \big)\bsb{I} - \bsb{\sigma}[\bsb{u},p]\cdot \nabla \bsb{v} - \bsb{\sigma}[\bsb{v},q]\cdot \nabla \bsb{u}. \label{eq:tensorE}
\end{equation}
The value of $\overstar{U}$ is determined by the material derivative of the no-net-force condition~\eqref{eq:NNF}, which reads
\begin{equation}
  \intG{ \overstar{\bsb{f}}, \bsb{e}_{z} } +   \intG{ \overstar{\bsb{f}} \, \divs \bsb{\theta}, \bsb{e}_{z} } = 0 \label{dd:NNF}
\end{equation}

The weak formulation~\eqref{eq:weakderivative} is obtained by applying the material differentiation identities \eqref{dd:IVS}, with the aid of formulas~\eqref{dd:grad}, to the variational equation~\eqref{eq:weakforward}, assuming the test functions in \eqref{eq:weakforward} to verify $\overstar{\bsb{v}}=\bsb{0}$, $\overstar{\bsb{g}}=\bsb{0}$ and $\overstar{q}=0$, i.e., to be convected under the shape perturbation (which the absence of boundary constraints in $\bsb{\mathcal{V}}$ allows). Finally, the Dirichlet BC \eqref{eq:uD} is used to set the right-hand side of (\ref{eq:weakderivative}b) to zero.

\proofstep{3. Material derivatives of energy functionals and drag force.} We subtract equation~\eqref{eq:weakderivative} with $(\bsb{v}, q, \bsb{g}) = (\bsb{u}, p, \bsb{f})$ from equation~\eqref{eq:weakforward} with $(\bsb{v}, q, \bsb{g}) = (\overstar{\bsb{u}}, \overstar{p}, \overstar{\bsb{f}})$ and use the no-net-force condition~\eqref{eq:NNF}, to obtain
\begin{equation}
  \intG{ u^{\sms{S}}\overstar{\bsb{\tau}} , \bsb{f} }
   - \intG{\overstar{\bsb{f}}, u^{\sms{S}}\bsb{\tau}} - \intG{\overstar{\bsb{f}}, \bsb{e}_z}\,U
   = - \intO{\bsb{E}\big((\bsb{u}, p),(\bsb{u}, p) \big), \nabla\tv \bsb{\theta}}
 + \intG{ \bsb{f} , (u^{\sms{S}}\bsb{\tau} + U\bsb{e}_z)\,\divs \bsb{\theta} }. \label{eq:diffint}
\end{equation}
We then use~\eqref{eq:diffint} in \eqref{eq:JWprime1}, which allows to eliminate the contribution of $\overstar{\bsb{f}}$ and yields
\begin{equation}
  J_{\sms{W}}'(\Gamma;\bsb{\theta})
 = 2\intG{ \bsb{f} , u^{\sms{S}} \overstar{\bsb{\tau}} }
 + \intO{\bsb{E}\big((\bsb{u}, p),(\bsb{u}, p) \big), \nabla\tv \bsb{\theta}}. \label{aux:01}
\end{equation}
Equality \eqref{eq:diffint} is also valid for the case where $\bsb{u}^{\sms{D}}=\bsb{e}_z$, i.e., $U=1, u^{\sms{S}}=0$ corresponding to the adjoint problem, since it does not rely on the no-net-force condition. Using this version in~\eqref{eq:F0prime} readily yields
\begin{equation}
  F_0'(\Gamma; \bsb{\theta}) = \intO{ \bsb{E}\big((\hat{\bsb{u}}, \hat{p}),(\hat{\bsb{u}}, \hat{p}) \big), \nabla\tv \bsb{\theta} }.
\label{aux:02}
\end{equation}
Lastly (and similarly), setting $(\bsb{v},q,\bsb{g}) = (\overstar{\bsb{u}}, \overstar{p}, \overstar{\bsb{f}})$ in the adjoint problem~\eqref{eq:weakadjoint} and $(\bsb{v},q,\bsb{g}) = (\hat{\bsb{u}}, \hat{p}, \hat{\bsb{f}})$ in (\ref{eq:weakderivative}a,b) of the derivative problem, then evaluating the combination (\ref{eq:weakderivative}a)+(\ref{eq:weakderivative}b)-\eqref{eq:weakadjoint}, provides
\begin{equation}
F_0U'(\Gamma;\bsb{\theta}) + \intG{\hat{\bsb{f}}, u^{\sms{S}} \overstar{\bsb{\tau}} }
= - \intO{ \bsb{E}\big((\bsb{u}, p),(\hat{\bsb{u}}, \hat{p}) \big), \nabla\tv \bsb{\theta}  }. \label{aux:03}
\end{equation}
The functional derivatives given by~\eqref{aux:01}-\eqref{aux:03} are thus free of any solution material derivatives.

\proofstep{4. Material derivatives of functionals: boundary-only form} The domain integrals involving $\bsb{E}$ are, in all three cases~\eqref{aux:01} to~\eqref{aux:03}, recast as follows in terms of only boundary integrals:
\begin{equation}
\begin{aligned}
  \text{(a)} &\qquad & J_{\sms{W}}' (\Gamma;\bsb{\theta})
 &= \int_{\Gamma} \big[ 2\mu \bsb{D}[\bsb{u}]\!:\!\bsb{D}[\bsb{u}]\theta_n + 2\bsb{f}\cdot( \overstar{\bsb{u}}{}^{\sms{S}} - \nabla\bsb{u}\cdot\bsb{\theta}) \big] \td S, \\
  \text{(b)} &\qquad & F_0' (\Gamma;\bsb{\theta})
 &= \int_{\Gamma} \big[ 2\mu \bsb{D}[\hat{\bsb{u}}]\!:\!\bsb{D}[\hat{\bsb{u}}]\theta_n - 2\hat{\bsb{f}}\cdot \nabla\hat{\bsb{u}}\cdot\bsb{\theta} \big] \, \td S, \\
  \text{(c)} &\qquad & F_0U' (\Gamma;\bsb{\theta})
 &= -\int_{\Gamma} \big[ 2\mu \bsb{D}[\bsb{u}]\!:\!\bsb{D}[\hat{\bsb{u}}]\theta_n - \bsb{f}\cdot \nabla\hat{\bsb{u}}\cdot\bsb{\theta}
 + \hat{\bsb{f}}\cdot ( \overstar{\bsb{u}}{}^{\sms{S}} - \nabla\bsb{u}\cdot\bsb{\theta}) \big] \td S.
\end{aligned}
\label{eq:derividentities}
\end{equation}
The proof of formulas~(\ref{eq:derividentities}a-c) rests on the following identity, established in Appendix~\ref{apd:proofs}:
\begin{lemma}\label{lemma:Etensor}
Let $(\bsb{u},p)$ and $(\bsb{v},q)$ both satisfy~\eqref{eq:forward}.
Then, for any vector field $ \bsb{\theta}\in C^{1,\infty}_{0}(\Omega)$, we have:
\begin{equation}
\intO{ \bsb{E} \big( (\bsb{u},p), (\hat{\bsb{u}},\hat{p})\big) , \nabla \tv \bsb{\theta} }
=
\int_{\Gamma} \bsb{n}\cdot \bsb{E}\big( (\bsb{u},p), (\hat{\bsb{u}},\hat{p})\big) \cdot \bsb{\theta} \,\td S.
\end{equation}
\end{lemma}
Formulas~(\ref{eq:derividentities}a-c) are then found by applying Lemma~\ref{lemma:Etensor} to the right-hand sides of~\eqref{aux:01}, \eqref{aux:02} and~\eqref{aux:03} and using the definition~\eqref{eq:tensorE} of $\bsb{E}$, with $(\bsb{v},q)=(\bsb{u},p)$ for~(\ref{eq:derividentities}a), $(\bsb{u},p)=(\bsb{v},q)=(\hat{\bsb{u}},\hat{p})$ for~(\ref{eq:derividentities}b) and $(\bsb{v},q)=(\hat{\bsb{u}},\hat{p})$ for~(\ref{eq:derividentities}c).

\proofstep{5. Express velocity normal derivatives using traction vectors.} The shape sensitivity formulas \eqref{eq:derividentities} remain somewhat inconvenient as they involve complete velocity gradients on $\Gamma$. This can be remedied by using the decomposition $\nabla\bsb{u}=\nabla_{\sms{S}}\bsb{u}+\pd_{n}\bsb{u}\otimes\bsb{n}$ of the velocity gradient (where $\nabla_{\sms{S}}\bsb{u}$ and $\pd_{n}\bsb{u}$ respectively denote the tangential gradient and the normal derivative of $\bsb{u}$) and expressing $\pd_n \bsb{u}$ in terms of $\bsb{f}$. In view of the specific form $\bsb{u}^{\sms{D}} = U\bsb{e}_z+u^{\sms{S}}\bsb{\tau}$ of the Dirichlet data on $\Gamma$, this step is carried out explicitly using curvilinear coordinates, and the following expressions are found by means of straightforward algebra (see the proof given in Appendix~\ref{apd:proofprop}) for the gradient and the strain rate tensor of the forward and adjoint solutions:
\begin{lemma}\label{lemma:grad}
Let $(\bsb{u}, p, \bsb{f})$ be the solution to the forward problem. On $\Gamma$,
\begin{equation}
\begin{aligned}
  \nabla\bsb{u}
 &= \frac{1}{2\mu}(f_n+p) \big( \bsb{n} \otimes \bsb{n} - \bsb{\tau} \otimes \bsb{\tau} \big)
  + \frac{R'u^{\sms{S}}}{R\alpha} \big( \bsb{\nu}\otimes\bsb{\nu} - \bsb{\tau} \otimes \bsb{\tau} \big)
  + \kappa u^{\sms{S}}(\bsb{n}\otimes\bsb{\tau}-\bsb{\tau} \otimes \bsb{n}) +\frac{1}{\mu}f_{\tau} \bsb{\tau} \otimes \bsb{n}, \\
  2\bsb{D}[\bsb{u}]
 &= \frac{1}{\mu}(f_n+p) \big( \bsb{n} \otimes \bsb{n} - \bsb{\tau} \otimes \bsb{\tau} \big)
  + \frac{2R'u^{\sms{S}}}{R\alpha} \big( \bsb{\nu}\otimes\bsb{\nu} - \bsb{\tau} \otimes \bsb{\tau} \big)
  + \frac{1}{\mu}f_{\tau} \big(\bsb{n}\otimes\bsb{\tau} +\bsb{\tau} \otimes \bsb{n}\big)
\end{aligned}
\label{eq:graduand2Du1}
\end{equation}
Similarly, for the solution $(\hat{\bsb{u}}, \hat{p}, \hat{\bsb{f}})$ of the adjoint problem, we have
\begin{equation}
\nabla \hat{\bsb{u}} = \frac{1}{\mu}\hat{f}_{\tau} \bsb{\tau}\otimes \bsb{n}, \quad 2\mu\bsb{D}[\hat{\bsb{u}}] = \hat{f}_{\tau} (\bsb{n}\otimes \bsb{\tau}+\bsb{\tau}\otimes \bsb{n}).
\end{equation}
\end{lemma}
In addition, as shown in Appendix~\ref{apd:proofs}, the material derivative $\overstar{\bsb{\tau}}$ of the unit tangent to generating arcs is given by
\begin{equation}
  \overstar{\bsb{\tau}} = \Big(\frac{1}{\alpha}\theta_n' + \kappa\theta_{\tau} \Big)\bsb{n}. \label{eq:taustar}
\end{equation}

The shape sensitivity formulas~(\ref{eq:JWprime}-c) are finally obtained by deriving explicit expressions of the right-hand sides of \eqref{eq:derividentities}. First, to establish~\eqref{eq:JWprime} for $J_{\sms{W}}'(\Gamma;\bsb{\theta})$, we derive the expressions
\begin{equation}
\begin{aligned}
  \bsb{f}\cdot\overstar{\bsb{u}}{}^{\sms{S}}
  &= u^{\sms{S}} f_n \big( \kappa \theta_{\tau} + \frac{1}{\alpha}\theta_n' \big), \\
  \bsb{f}\cdot\nabla\bsb{u}\cdot\bsb{\theta}
  &= \Big( \frac{1}{\mu} f_{\tau}^2 + \frac{1}{2\mu}(f_n+p)f_n - \kappa u^{\sms{S}} f_{\tau} \Big) \theta_n
  + \Big( \frac{1}{\alpha}(u^{\sms{S}} )' f_{\tau} + \kappa u^{\sms{S}} f_n \Big) \theta_{\tau} \\
  2\mu \bsb{D}[\bsb{u}] \!:\! \bsb{D}[\bsb{u}] \theta_n
  &= \Big[ \frac{1}{\mu} f_{\tau}^2 + 4\mu \Big( \frac{R'u^{\sms{S}}}{R\alpha} \Big)^2
  + \frac{1}{\mu}(f_n+p)^2 \Big] \theta_n
\end{aligned}
\end{equation}
with the help of~\eqref{eq:taustar} and~\eqref{eq:graduand2Du1}. After rearrangement, we obtain
\begin{multline} \label{aux:04}
  2\mu \bsb{D}[\bsb{u}] \!:\! \bsb{D}[\bsb{u}] \theta_n
  + 2\bsb{f}\cdot(\overstar{\bsb{u}}{}^{\sms{S}}-\nabla\bsb{u}\cdot\bsb{\theta}) \\
 = \Big[ -\frac{1}{\mu} f_{\tau}^2 + 4\mu \Big( \frac{R'u^{\sms{S}}}{R\alpha} \Big)^2
  + \frac{1}{\mu}(f_n+p)p + 2\kappa u^{\sms{S}} f_{\tau} \Big] \theta_n %\\ & \qquad
  -\frac{2}{\alpha}f_{\tau}\theta_{\tau}(u^{\sms{S}} )' + \frac{2}{\alpha}u^{\sms{S}} f_n\theta_n '.
\end{multline}
We then similarly derive the formulas
\begin{align}
\MoveEqLeft[13]
  2\mu  \bsb{D}[\bsb{u}] \!:\! \bsb{D}[\hat{\bsb{u}}]  \theta_n
  + \hat{\bsb{f}}\cdot (\overstar{\bsb{u}}{}^{\sms{S}} - \nabla\bsb{u}\cdot\bsb{\theta})
  - \bsb{f} \cdot \nabla \hat{\bsb{u}} \cdot \bsb{\theta} \\
 &= \Big[ \kappa u^{\sms{S}} \hat{f}_{\tau} - \frac{1}{\mu} f_{\tau}\hat{f}_{\tau} + \frac{1}{2\mu}(f_n+p)\hat{p} \Big] \theta_n \label{aux:05} %\\ & \qquad
  - \frac{1}{\alpha}\hat{f}_{\tau}\theta_{\tau}(u^{\sms{S}} )' - \frac{1}{\alpha}u^{\sms{S}} \hat{p}\theta_n ' \\
  2\mu \bsb{D}[\hat{\bsb{u}}]\!:\!\bsb{D}[\hat{\bsb{u}}]\theta_n - 2\hat{\bsb{f}}\cdot \nabla\hat{\bsb{u}}\cdot\bsb{\theta}
 &= -\frac{1}{\mu} \hat{f}_{\tau}^2 \theta_n, \label{aux:06}
\end{align}
having in particular used that $\hat{f_n} = - \hat{p}$. The sought material derivative formulas~(\ref{eq:JWprime}-c) finally follow by using identities \eqref{aux:04}, \eqref{aux:05} and \eqref{aux:06} in formulas~(\ref{eq:derividentities}a-c)

\proofstep{6. End of proof} Steps 1 to 5 above complete the proof of identities~(\ref{eq:JWprime}-c), and formula~\eqref{eq:JDprime} follows directly from the definition of $J_{\text{drag}}$ given after~\eqref{eq:opt_drag}.

%%%%% NUMERICAL %%%%%%%
\section{Numerical method}
\label{sc:NumMeth}

In this section, we describe the main numerical methods employed for solving the shape optimization problem, utilizing the shape sensitivity formulas derived in the preceding sections.

\subsection{Stokes PDE solver}
The shape optimization problem requires solution of three different boundary value problems: (i) the {\em forward problem} with given slip velocity~(\ref{eq:forward}--\ref{eq:NNF}), (ii) the {\em adjoint problem} with given Dirichlet conditions~\eqref{eq:adj0}, and (iii) the {\em auxiliary problem} with mixed boundary conditions~\eqref{eq:optslip} for determining the optimal slip for a given shape.   
Similar to our previous work~\cite{guo2021slip,guo2021disp} which studied the optimization of slip velocity profiles and ciliary locomotion for a given shape $\Gamma$, we employ an indirect boundary integral equation (BIE) formulation for solving these three problems. 

Specifically, we start from the single-layer potential {\em ansatz}, which expresses the velocity $\bsb{u}$ as a convolution of an unknown axially-symmetric density function $\bsb{\zeta}$ defined on $\Gamma$ with the Green's function for the Stokes equations:
\begin{equation}
  \bsb{u}(\bsb{x})
  = \mathcal{S}[\bsb{\zeta}] (\bsb{x}), \quad \text{where}\quad  \mathcal{S}[\bsb{\zeta}](\bsb{x}) = \frac{1}{8\pi\mu}\int_\Gamma \left(\frac{1}{|\bsb{r}|} \mathbf{I} +  \frac{\bsb{r} \otimes \bsb{r} }{|\bsb{r}|^3}\right) \, \bsb{\zeta}\left(\bsb{y}\right) \, \mathrm{d}S \quad\text{and}\quad \bsb{r} = \bsb{x} - \bsb{y}.\label{u:SL} \end{equation}
 The ansatz \eqref{u:SL} satisfies the Stokes PDE by construction and taking the limit as $\bsb{x}$ approaches $\Gamma$ from the exterior and applying boundary conditions for each of the three problems results in a set of BIEs for the unknown density $\bsb{\zeta}$. The traction vector $\bsb{f}$ and pressure density $p$ on the surface $\Gamma$ can be evaluated as a convolution of $\bsb{\zeta}$ with the traction and the pressure kernels respectively:
 \begin{align}
  \bsb{f}(\bsb{x}) 
 &=-\frac{1}{2}\bsb{\zeta}\left(\bsb{x}\right) + \mathcal{K} [\bsb{\zeta}](\bsb{x}), \quad\text{where}\quad \mathcal{K} [\bsb{\zeta}] (\bsb{x})=  \frac{3}{4\pi}\int_{\Gamma} \left(\frac{\bsb{r} \otimes \bsb{r} }{|\bsb{r}|^5}\right) (\bsb{r}\cdot\bsb{n}\left(\bsb{x}\right))\bsb{\zeta}\left(\bsb{y}\right)\, \mathrm{d}S, \label{f:SL} \\
 p(\bsb{x}) 
 &=-\frac{1}{2}\bsb{\zeta}\left(\bsb{x}\right)\cdot\bsb{n}\left(\bsb{x}\right) + \frac{1}{4\pi}\int_{\Gamma} \frac{\bsb{r} \cdot \bsb{\zeta}\left(\bsb{y}\right) }{|\bsb{r}|^3}\, \mathrm{d}S. \label{p:SL}
\end{align}

In the case of the adjoint problem, denoting the unknown density by $\hat{\bsb{\zeta}}$, substituting the ansatz in the Dirichlet boundary condition in~\eqref{eq:adj0} and then evaluating the traction by~\eqref{f:SL}, we get
\begin{equation} 
    \mathcal{S}[\hat{\bsb{\zeta}}](\bsb{x}) 
 = \bsb{e}_z, \qquad \hat{\bsb{f}}=-\frac{1}{2}\hat{\bsb{\zeta}}+\mathcal{K}[\hat{\bsb{\zeta}}]
 \qquad \text{\em (BIE for adjoint problem~\eqref{eq:adj0}).}
\label{eq:fullsys2}
\end{equation}
Then, the mixed boundary conditions featured in the auxiliary problem~\eqref{eq:optslip} yield
\begin{equation}\label{eq:fullsys3} 
\begin{aligned}
        \Big(-\frac{1}{2} + \mathcal{K} \Big)[\Tilde{\bsb{\zeta}}] \cdot \bsb{\tau} = \hat{\bsb{f}}\cdot\bsb{\tau} & \\
           \mathcal{S}[\Tilde{\bsb{\zeta}}]\cdot\bsb{n}= 0 &
  \end{aligned} \qquad \text{\em (BIEs for auxiliary problem~\eqref{eq:optslip}).}
\end{equation}
Similarly, for the forward problem~\eqref{eq:uD}, the unknowns $\bsb{\zeta}$ and $U$ can be obtained by substituting~\eqref{u:SL} in~\eqref{eq:uD} and~\eqref{f:SL} in~\eqref{eq:NNF}, yielding for any $\bsb{x}$ on $\Gamma$,
\begin{equation}\label{eq:fullsys} 
\begin{aligned}
    \mathcal{S}[\bsb{\zeta}](\bsb{x}) - U\bsb{e}_z =  \bsb{u}^{\sms{S}} (\bsb{x}) &\\
        \Big<\Big(-\frac{1}{2} + \mathcal{K} \Big)[\bsb{\zeta}] \, , \,\bsb{e}_z \Big>_\Gamma = 0 &
  \end{aligned} \qquad \text{\em (BIEs for the forward problem~\eqref{eq:uD}).}
\end{equation}
We convert the weakly singular boundary integrals in~\cref{eq:fullsys,eq:fullsys2,eq:fullsys3} into convolutions on the generating curve $\gamma$ by performing an analytic integration in the orthoradial direction, and applying a high-order quadrature rule designed to handle the $log-$singularity of the resulting kernels~\cite{veerapaneni2009numerical}. In addition, shape sensitivity formulas in Proposition~\ref{prop:Jw} require evaluating the pressure fields $p$ and $\hat{p}$ on the particle surface $\Gamma$. We make use of a generalized Gaussian quadrature rule, developed in~\cite{bremer2010nonlinear}, for accurate numerical integration of the exhibited strong $r^{-2}-$singularity of the kernel in~\eqref{p:SL}.

\subsection{Finite parametrizations of slip velocity and shape}\label{sc:finite}
We employ fifth-order B-splines to parametrize the unknowns of the optimization problems~\eqref{eq:const_opt} and~\eqref{eq:opt_drag}, namely the scalar slip velocity profile $u^{\sms{S}}$ and the functions $R,Z$ that define $\Gamma$ through ~\eqref{wall:shape}), by
\begin{equation}
u^{\sms{S}}(t) = \bsb{\xi}_{u}\tv \bsb{w}(t) , \quad  R(t)= \bsb{\xi}_R \tv \bsb{B}(t), \quad Z(t)= \bsb{\xi}_Z\tv \bsb{B}(t),  \quad t\in[0,\pi].  \label{eq:finiteparamvec}
\end{equation}
The vector-valued functions $\bsb{w}(t)$ and $\bsb{B}(t)$ are provided in Appendix \ref{apd:bsp}. The design vectors for $u^{\sms{S}}$ and $\gamma$ are denoted by $\bsb{\xi}_{u} = \big( \xi^1_{u}, \dots, \xi^{N\!_{u}}_{u} \big)\tv $ and $\bsb{\xi}_{\gamma} = \big( \bsb{\xi}_R,\bsb{\xi}_Z \big) = \big( \xi^1_{R},\dots, \xi^{N\!_{R}}_{R}, \xi^1_{Z},\dots,\xi^{N\!_{Z}}_{Z} \big)\tv $, respectively.
In order to satisfy the constraints~(\ref{eq:rz}a,c), the shape design vector $\bsb{\xi}_{\gamma}$ has $N\!_R+N\!_Z-4$ degrees of freedom, as $\xi^1_{R}, \xi^{N\!_{R}}_{R}, \xi^1_{Z}, \xi^{N\!_{Z}}_{Z}$ are determined upon other entries in $\bsb{\xi}_{\gamma}$ (see Appendix \ref{apd:bsp}).

Then, transformation velocities $\bsb{\theta}$ associated with the parametrization~\eqref{eq:finiteparamvec} may be defined as
\begin{equation}
  \bsb{\theta}
 = \bsb{\zeta}_R \tv \bsb{B}(t) \bsb{e}_r + \bsb{\zeta}_Z\tv \bsb{B}(t)\bsb{e}_z \label{theta:discr} 
\end{equation}
where $(\bsb{\zeta}_R$, $\bsb{\zeta}_Z)=:\bsb{\zeta}_{\gamma}$ define perturbation directions for the shape parameter vectors $\bsb{\xi}_R$, $\bsb{\xi}_Z$ which must be consistent with~(\ref{eq:rz}a,c) and are otherwise arbitrary.

\subsection{Numerical optimization scheme}\label{sc:NumSch}
 
% figure
\begin{figure}[h] \centering
  \includegraphics[width=0.7\textwidth]{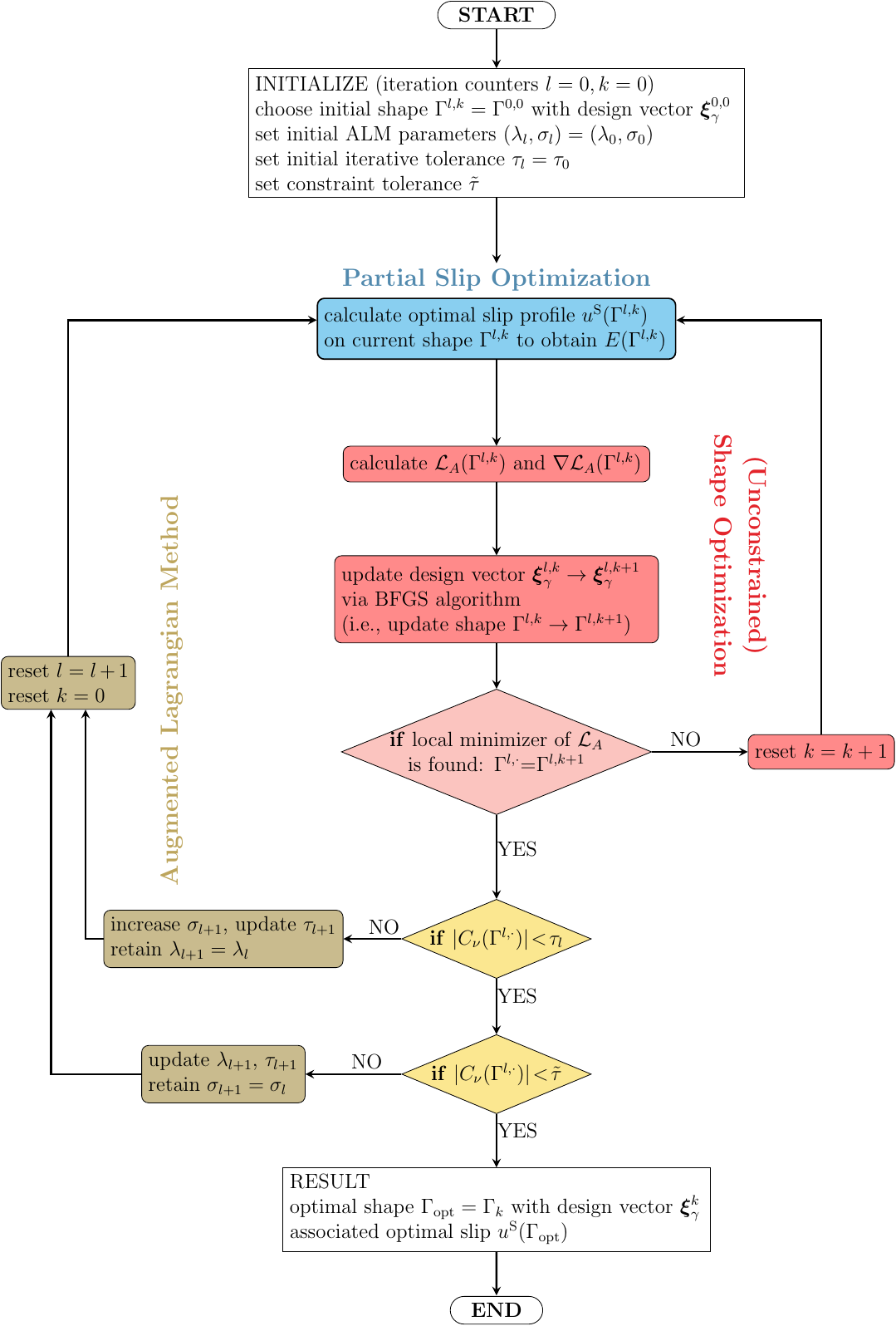}
  \caption{{\em Framework of the optimization algorithm for the maximum swimming efficiency. The algorithm consists of three main parts shown in different colors. The single block in blue is the slip optimization as described in Section \ref{sec:slipopt}. It is worth noting that the slip optimization process is nested in the shape optimization process (in red/pink). The process in yellow/brown color briefly demonstrate the algorithm of the Augmented Lagrangian Method (ALM) in Section \ref{sc:NumSch}.}}
  \label{fig:flow}
\end{figure}

We use the augmented Lagrangian method (ALM) to adapt the constrained optimization problem \eqref{eq:const_opt} to a sequence of unconstrained problems~\cite[Chapter~17]{wright2006numerical} of the form
\begin{equation}
\min\limits_{\Gamma} \mathcal{L}_A(u^{\text{S}}, \Gamma) \quad \text{with} \quad  \mathcal{L}_A(\Gamma) = - E(u^{\sms{S}}, \Gamma) - \lambda_l C_{\nu}(\Gamma) + \frac{\sigma_l}{2} C_{\nu}{}^2(\Gamma), \qquad l=0,1,2,\ldots
\label{eq:effi_opt_alm}
\end{equation}
The variable $\lambda_l$ is an explicit estimate of the Lagrange multiplier, and $\sigma_l$ is a penalty parameter. The optimization algorithm fixes the values of $\lambda_l$ and $\sigma_l$ at the $l$th iteration and performs minimization for $\mathcal{L}_A$. Fig.~\ref{fig:flow} depicts the framework of the optimization problem \eqref{eq:effi_opt_alm} as the slip optimization is implemented internally.
In this flowchart, the shape $\Gamma$ and its design vector $\bsb{\xi}_{\gamma}$ are often accompanied by superscripts $l,k$. The first number $l$ indicates the iteration of ALM when updating $\lambda_l$, $\sigma_l$, and the second number indicates the iteration in the unconstrained shape optimization problem \eqref{eq:effi_opt_alm} with specific $l$.
The optimization process starts with an arbitrary initial shape $\Gamma^{0,0}$ given by the design vector $\bsb{\xi}_{\gamma}^{0,0}$.
The optimal slip profile for the current shape is calculated directly as described in Section \ref{sec:slipopt}, shown as a single blue block in Fig.~\ref{fig:flow}. The shape is then updated via the BFGS algorithm for the unconstrained problem \eqref{eq:effi_opt_alm}, which makes use of the shape derivative of $E(\Gamma)$ given by Proposition~\ref{prop:dE}, and the slip profile must be recalculated in every iteration of the BFGS algorithm. When the optimal shape is found, the process moves to parameter updating of ALM (if tolerance criteria are not satisfied), then a new unconstrained problem \eqref{eq:effi_opt_alm} is defined and implemented. If the local minimizer of \eqref{eq:effi_opt_alm} for current $l$ satisfies the tolerance criteria, the optimization is completed and outputs the optimal shape.

Similarly, the constrained drag force minimization problem~\eqref{eq:opt_drag} is solved using the sequence of unconstrained problems
%we simply consider the unconstrained problem
%\begin{equation}
%\min\limits_{\Gamma} J_{\text{drag}}.
%\end{equation}
\begin{equation}
    \min\limits_{\Gamma} \mathcal{L}_A(\Gamma) \quad \text{with} \quad  \mathcal{L}_A(\Gamma) = J_{\text{drag}} - \lambda_l C_{\nu}  + \frac{\sigma_l}{2} C_{\nu}{}^2, \qquad l=0,1,2,\ldots.
    \label{eq:drag_opt_alm}
\end{equation}
The optimization is implemented similarly as Fig.~\ref{fig:flow}, except that the partial slip optimization step is omitted. Solving the unconstrained problem \eqref{eq:drag_opt_alm} makes use of~\eqref{eq:JDprime} for the shape derivative of the drag force. It does not involve any slip velocity, thus, the optimization starts with calculating $\mathcal{L}_A$ and $\nabla\mathcal{L}_A$ in \eqref{eq:drag_opt_alm} directly. The rest of the process remains the same as before. \enlargethispage*{1ex}

\begin{table}[b]
\centering
    \begin{tabular}{cccccc}\hline
      Initial $\Gamma$ & Perturbed $\Gamma_\eta\big|_{\eta=1}$ & $E'$ abs.err. & $E'$ rel.err. & $J'_{\sms{drag}}$ abs.err. & $J'_{\sms{drag}}$ rel.err. \\ 
      \raisebox{-0.5\height}{\includegraphics[height=40pt]{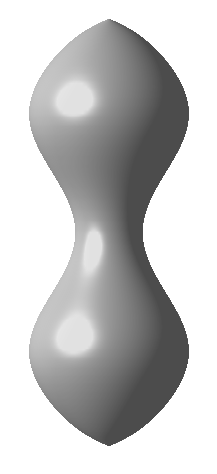}} & \raisebox{-0.5\height}{\includegraphics[height=40pt]{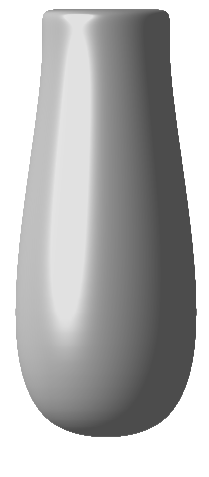}} & \num[round-mode=places,round-precision=2]{3.970360e-07} &  \num[round-mode=places,round-precision=2]{8.452474e-07} &  \num[round-mode=places,round-precision=2]{5.445094e-08} & \num[round-mode=places,round-precision=2]{5.230632e-07}  \\ 
      \raisebox{-0.5\height}{\includegraphics[height=40pt]{./FigUpdated/shape_0_peanut}} & \raisebox{-0.5\height}{\includegraphics[height=40pt]{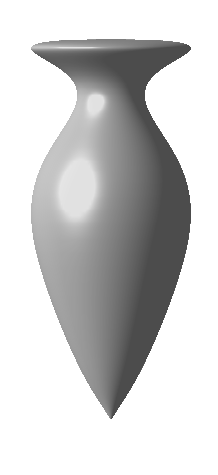}} & \num[round-mode=places,round-precision=2]{2.673148e-08} & \num[round-mode=places,round-precision=2]{2.142554e-07} & \num[round-mode=places,round-precision=2]{2.076638e-09} & \num[round-mode=places,round-precision=2]{8.223794e-08} \\ 
      \raisebox{-0.5\height}{\includegraphics[height=40pt]{./FigUpdated/shape_0_peanut}} & \raisebox{-0.5\height}{\includegraphics[height=40pt]{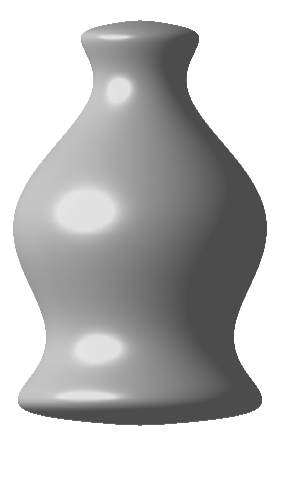}} & \num[round-mode=places,round-precision=2]{8.135964e-07} & \num[round-mode=places,round-precision=2]{8.117251e-07}  & \num[round-mode=places,round-precision=2]{2.728198e-07} & \num[round-mode=places,round-precision=2]{1.120857e-06} \\ \hline 
      \end{tabular}
      \captionof{table}{{\em Comparison of shape sensitivities obtained by analytic formulas and by the central finite difference scheme. The absolute errors (abs.err.) is $| J' - J'_{FD}|$, where $J'$ is the result of the analytic sensitivity formula (either~\eqref{eq:JDprime} or~\eqref{eq:ddE}), and $J'_{FD}$ is the result of \eqref{eq:JprimeFD}. The relative error (rel.err.) is calculated by $| (J' - J'_{FD})/J'_{FD}|$. Note that the perturbed shapes are shown at $\eta = 1$ for visualization purpose to highlight the direction of perturbation.}}\label{tab:perturb}
\end{table}
%%%%%%%%%%%%%%%%%%%%%%%%%%%%%%%%%%%%%%%%%%%%%%%%%%%%%%%%%%%%%%%%%%%%%%%%%%%%%%

\section{Results and discussion}
\label{sc:results}

In this section, we first validate the shape sensitivity formulas derived in Section~\ref{sec:shape_sens}. Then, we present several numerical experiments to demonstrate the shape optimization approach for microswimmers and analyze the optimal shapes obtained for various reduced volumes. Additionally, we compare the results with a simple drag minimization problem to highlight the differences in different configurations under different objectives: maximizing efficiency or minimizing drag. \enlargethispage*{5ex}

\subsection{Verification of shape sensitivities} Here, we validate the shape sensitivity formulas~\eqref{eq:JDprime} and \eqref{eq:ddE} by comparing them with the numerical approximations via the finite difference method. We use the central difference formula for comparison, given by
\begin{equation}\label{eq:JprimeFD}
    J'_{\sms{FD}} = \frac{J\big(\Gamma(\bsb{\xi}_{\gamma} + \eta \bsb{\zeta}_{\gamma})\big) - J\big(\Gamma(\bsb{\xi}_{\gamma} - \eta \bsb{\zeta}_{\gamma})\big)}{2\eta}
\end{equation}
where $J$ is either $E(\Gamma)$ or $J_{\text{drag}}(\Gamma)$ and $\bsb{\zeta}_{\gamma}$ is a shape parameter perturbation direction that defines a transformation velocity $\bsb{\theta}$ through~\eqref{theta:discr}.
% The change of design vector, $\Delta\bsb{\xi}_{\gamma}$, is calculated based on the transformation velocity field $\bsb{\theta}$. In particular, the test shape is given by $\boldsymbol{x} \in \Gamma(\boldsymbol{\xi}_{\gamma})$, after the perturbation, the shape becomes $(\boldsymbol{x} + \boldsymbol{\theta}) \in \Gamma(\boldsymbol{\xi}_{\gamma} + \Delta\boldsymbol{\xi}_{\gamma})$. This indicates a one-to-one relation between $\boldsymbol{\theta}$ and $\Delta\boldsymbol{\xi}_{\gamma}$. Moreover, this relation is linear, i.e., $(\boldsymbol{x} + \eta \boldsymbol{\theta}) \in \Gamma(\boldsymbol{\xi}_{\gamma} + \eta \Delta\boldsymbol{\xi}_{\gamma})$, as a direct result of Section \ref{sc:finite}.

We choose an arbitrary initial shape and then perturb it into a variety of other arbitrary shapes. Table~\ref{tab:perturb} lists the absolute and relative errors between the analytic formula evaluations and their finite difference approximations, with the step size $\eta=10^{-3}$ in~\eqref{eq:JprimeFD} for all cases. The error results validate the correctness of the analytic shape sensitivity formulas given in Section~\ref{sec:shape_sens}.\enlargethispage*{1ex}

\subsection{Optimization results} First, we showcase the iterative shape optimization process towards a maximal efficiency shape  in 
Fig.~\ref{fig:shapeiterations3}, starting from an arbitrary peanut-shape. The flow field snapshots around the microswimmer driven by the optimal slip profile that maximizes the swimming efficiency are demonstrated during the optimization process.
In particular, a peanut-shape microswimmer with $\nu = 0.7$ is used as the initial shape. The high-velocity region can be observed around the entire swimmer in early iterations. This high-velocity region, in turn, leads to a low swimming efficiency at $J_{\sms{E}} = 53\%$.
During the optimization process, the swimmer first transitions from its original concave shape to a convex shape that resembles a prolate spheroidal shape with blunt poles (iteration 12), and then ``sharpens'' the poles in the next few iterations (iteration 17), leading to a swimming efficiency as high as $J_E=332\%$. We note that the high-velocity region is gradually reduced during the optimization, and the reduction progresses from the equatorial region toward the poles. \enlargethispage*{3ex}

\begin{figure}[t] \centering
  \includegraphics[width=0.92\textwidth,trim={0pt 0pt 0pt 0pt},clip]{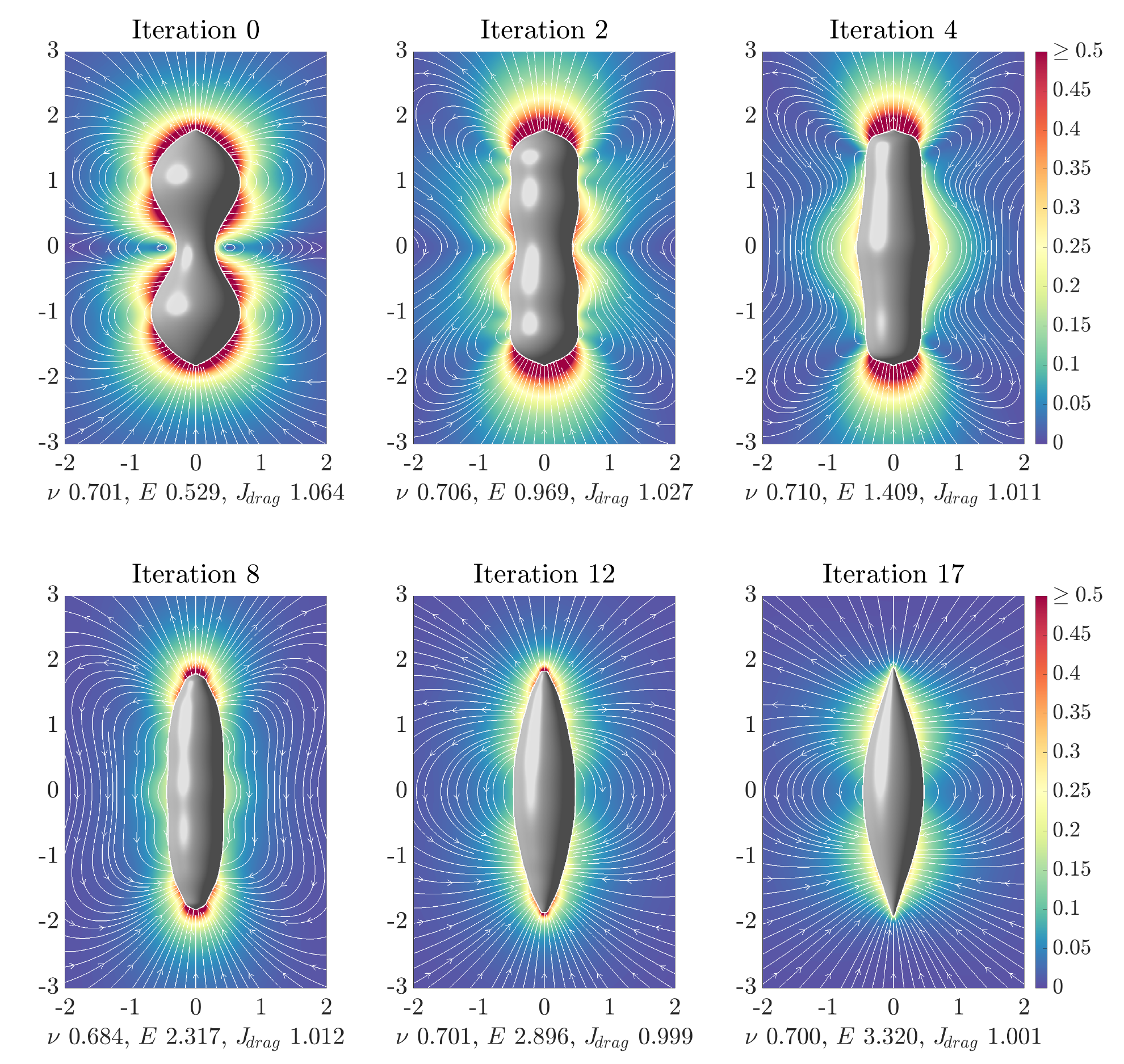}
  \caption{{\em Snapshots from the shape optimization of an initially peanut-shaped microswimmer with reduced volume, $\nu_0=0.7$. Here, the fluid velocity is shown in the lab (fixed) frame and the propulsion velocity $U$ is scaled to one.}}\label{fig:shapeiterations3}
\end{figure}
% just one example
%We present our numerical optimization results in this section.
% Fig.~\ref{fig:EFFvsRV} compares the optimal shapes of the microswimmers and their maximum efficiency under the constraint of various values of the reduced volume. The optimal shapes for different reduced volumes always exhibit a pointy-tip configuration. The size of the pointy-tip decreases as the reduced volume increases. %As the reduced volume equates one, the shape must be a perfect sphere. 
% The drag minimization shapes and corresponding maximum efficiency are also shown in Fig.~\ref{fig:EFFvsRV}, while the value of efficiency is between that of the optimal shape and that of the spherical prolate. 
% %(\textcolor{red}{double check using initial shape from drag minimization when $\nu$=0.9})

% We obtain the optimal shapes with the associated optimal slip profiles for a range of reduced volume ($0.6 \le \nu \le 1.0$).
Next, we consider several initial shapes whose reduced volumes are in the range $0.6 \le \nu \le 1.0$ and optimize their shape and associated slip profiles.
The swimming efficiency corresponding to the optimal shape and slip is referred to as the maximal swimming efficiency, and is shown in Fig~\ref{fig:JvsRV}(a) in orange stars as a function of reduced volume. 
% By definition, $\nu = 1$ restricts the shape to a sphere and thus prevents any shape optimization. 
Since the only possible shape for $\nu = 1$ is sphere, 
the optimization problem reduces to finding the optimal slip profile of a sphere to maximize the swimming efficiency. In this case, we recover the standard result that the optimal profile is a sine function $u^S = \sin(t)$ that yields the swimming efficiency $J_E = 50\%$~\cite{michelin2010efficiency}.
As $\nu$ decreases, the shapes of the microswimmers deviate from sphere. The maximal swimming efficiency monotonically increases with the decrease of $\nu$. 
For all reduced volumes we test, the optimal shapes that maximize the swimming efficiency are pointy elongated front-back symmetric shapes as shown in the insets of Fig.~\ref{fig:JvsRV}(a).
As references, we optimize the slip profiles on two other shape families and compare the swimming efficiencies against the maximal swimming efficiencies. Specifically, we consider prolate spheroids and shapes that minimize the fluid drag for given reduce volumes $\nu$. The swimming efficiencies and the corresponding shapes are presented in Fig.~\ref{fig:JvsRV}(a). 
To obtain the swimming efficiencies corresponding to these shape families, we apply the partial optimization method with fixed shape described in Section~\ref{sec:slipopt}.
Notably, for any given reduced volume, the prolate spheroid always underperforms its two counterparts (the shape that maximizes the swimming efficiency and the shape that minimizes the drag force). The swimming efficiencies of the shapes that minimize fluid drag are no more than $10\%$ worse than those of the optimal shapes when $0.8 \le \nu \le 1.0$, but become less competitive when $\nu$ is further decreased. For example, when $\nu = 0.60$, the swimming efficiencies of the prolate spheroid and the shape that minimizes fluid drag are $386\%$ and $480\%$ respectively, significantly lower than $580\%$ obtained by the optimal shape.

Unlike the maximal swimming efficiency, the drag force required to tow a rigid body along the axis of symmetry at unit speed does not vary monotonically with the reduced volume. 
Fig.~\ref{fig:JvsRV}(b) shows the drag force normalized by the force required to tow a rigid sphere of the same volume.
The drag force is minimized when $\nu \approx 0.90$, in which case the normalized drag force is approximately $0.95$, consistent with the classical results~\cite{pironneau1973optimum,bourot1974numerical}. 
Further decreasing the reduced volume increases the fluid drag. In fact, our results show that shapes with $\nu < 0.70$ experience higher fluid drag compare to the spheres of the same volume.
That being said, unlike the drastic effect of reduced volume on the swimming efficiency, the change in the fluid drag resulted from the change in shape is rather moderate, ranging from the low of $0.95$ to the high of $1.06$. 
The values of swimming efficiencies and fluid drag for these three shape families at different reduced volumes are shown in Table~\ref{tab:compare_all}.

\begin{figure}[h] \centering
  \includegraphics[width = 0.95\textwidth]{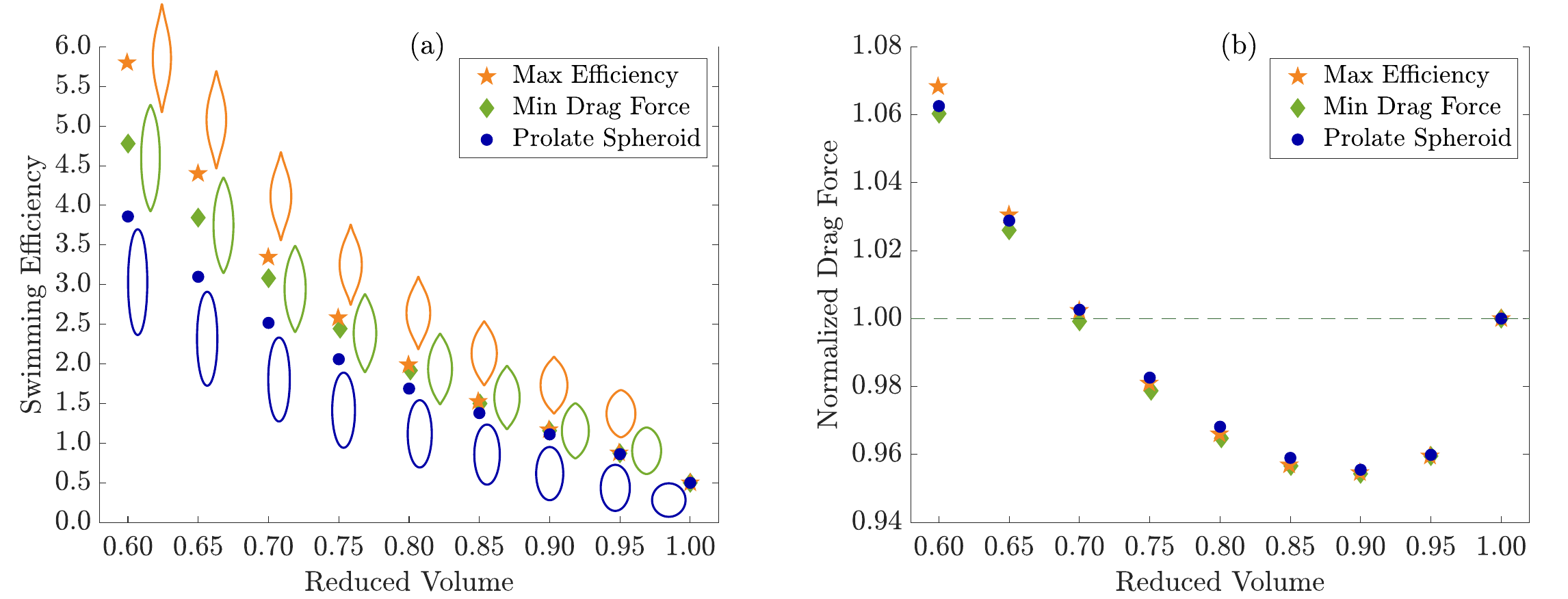}
  \caption{{\em (a) Swimming efficiency versus reduced volume of different shape families. The shapes and the slip profiles that maximize the swimming efficiencies are obtained using algorithms detailed in Fig.~\ref{fig:flow}. The swimming efficiencies for the prolate spheroid and the shape that minimizes the fluid drag are obtained by the slip optimization algorithm detailed in Section~\ref{sec:slipopt} while the body shapes are fixed.
  (b) Normalized drag force versus reduced volume of different shape families. No significant difference is found in the fluid drag force between the three shape families for any given reduced volume, as long as the shape is similar to elongated front-back symmetric prolate spheroids.}} \label{fig:JvsRV}
\end{figure}

\begin{table}[t]
    \centering
    \addtolength{\tabcolsep}{6pt}   
    \begin{tabular}{ccccc}
    \toprule %\\%[-24pt]
    Reduced Volume & & \multicolumn{3}{c}{Microswimmer Shape Type} \\%[-4pt]
    \cmidrule(r){3-5} %\\%[-24pt]
    Constraint &  & Prolate Spheroid & Max Efficiency* & Min Drag Force \\%[-4pt] 
    \midrule %\\%[-24pt]
    \multirow{3}{*}{$\nu_0 = 0.60$} & $\nu$ & \num[round-mode=places,round-precision=3]{0.600000} & \num[round-mode=places,round-precision=3]{0.599398}  & \num[round-mode=places,round-precision=3]{0.600115} \\%[-8pt] 
    & $J_{\sms{E}}$ & \num[round-mode=places,round-precision=3]{3.859753} & \num[round-mode=places,round-precision=3]{5.800658} & \num[round-mode=places,round-precision=3]{4.779075} \\%[-8pt] 
    & $J_{\text{drag}}$ & \num[round-mode=places,round-precision=3]{1.062557} & \num[round-mode=places,round-precision=3]{1.068315} & \num[round-mode=places,round-precision=3]{1.060323} \\%[-3pt]
    \midrule %\\%[-24pt]
    \multirow{3}{*}{$\nu_0 = 0.65$} & $\nu$ & \num[round-mode=places,round-precision=3]{0.650000} & \num[round-mode=places,round-precision=3]{0.649839} & \num[round-mode=places,round-precision=3]{0.649966}  \\%[-8pt]  
    & $J_{\sms{E}}$ & \num[round-mode=places,round-precision=3]{3.099078} & \num[round-mode=places,round-precision=3]{4.401910} & \num[round-mode=places,round-precision=3]{3.844991} \\%[-8pt]  
    & $J_{\text{drag}}$ & \num[round-mode=places,round-precision=3]{1.028850} & \num[round-mode=places,round-precision=3]{1.030539} & \num[round-mode=places,round-precision=3]{1.026001} \\%[-3pt]
    \midrule %\\%[-24pt]
    \multirow{3}{*}{$\nu_0 = 0.70$} & $\nu$ & \num[round-mode=places,round-precision=3]{0.700000} & \num[round-mode=places,round-precision=3]{0.699824} & \num[round-mode=places,round-precision=3]{0.700048} \\%[-8pt]  
    & $J_{\sms{E}}$ & \num[round-mode=places,round-precision=3]{2.517108} & \num[round-mode=places,round-precision=3]{3.346512} & \num[round-mode=places,round-precision=3]{3.081304} \\%[-8pt]  
    & $J_{\text{drag}}$ & \num[round-mode=places,round-precision=3]{1.002602} & \num[round-mode=places,round-precision=3]{1.002411} & \num[round-mode=places,round-precision=3]{0.999188} \\%[-3pt]
    \midrule %\\%[-24pt]
    \multirow{3}{*}{$\nu_0 = 0.75$} & $\nu$ & \num[round-mode=places,round-precision=3]{0.750000} & \num[round-mode=places,round-precision=3]{0.749533} & \num[round-mode=places,round-precision=3]{0.750949} \\%[-8pt]  
    & $J_{\sms{E}}$ & \num[round-mode=places,round-precision=3]{2.059092} & \num[round-mode=places,round-precision=3]{2.582476} & \num[round-mode=places,round-precision=3]{2.444336} \\%[-8pt]  
    & $J_{\text{drag}}$ & \num[round-mode=places,round-precision=3]{0.982620} & \num[round-mode=places,round-precision=3]{0.981012} & \num[round-mode=places,round-precision=3]{0.978758} \\%[-3pt] 
    \midrule %\\%[-24pt]
    \multirow{3}{*}{$\nu_0 = 0.80$} & $\nu$ & \num[round-mode=places,round-precision=3]{0.800000} & \num[round-mode=places,round-precision=3]{0.799544} & \num[round-mode=places,round-precision=3]{0.800934} \\%[-8pt]  
    & $J_{\sms{E}}$ & \num[round-mode=places,round-precision=3]{1.688483} & \num[round-mode=places,round-precision=3]{1.988867} & \num[round-mode=places,round-precision=3]{1.916884} \\%[-8pt]  
    & $J_{\text{drag}}$ & \num[round-mode=places,round-precision=3]{0.968175} & \num[round-mode=places,round-precision=3]{0.966074} & \num[round-mode=places,round-precision=3]{0.964769} \\%[-3pt]
    \midrule %\\%[-24pt]
    \multirow{3}{*}{$\nu_0 = 0.85$} & $\nu$ & \num[round-mode=places,round-precision=3]{0.850000} & \num[round-mode=places,round-precision=3]{0.849157} & \num[round-mode=places,round-precision=3]{0.850410} \\%[-8pt]  
    & $J_{\sms{E}}$ & \num[round-mode=places,round-precision=3]{1.379398} & \num[round-mode=places,round-precision=3]{1.527532} & \num[round-mode=places,round-precision=3]{1.499798} \\%[-8pt]  
    & $J_{\text{drag}}$ & \num[round-mode=places,round-precision=3]{0.959001} & \num[round-mode=places,round-precision=3]{0.956948} & \num[round-mode=places,round-precision=3]{0.956629} \\%[-3pt]
    \midrule %\\%[-24pt]
    \multirow{3}{*}{$\nu_0 = 0.90$} & $\nu$ & \num[round-mode=places,round-precision=3]{0.900000} & \num[round-mode=places,round-precision=3]{0.899427} & \num[round-mode=places,round-precision=3]{0.899949} \\%[-8pt]  
    & $J_{\sms{E}}$ & \num[round-mode=places,round-precision=3]{1.111298} & \num[round-mode=places,round-precision=3]{1.168898} & \num[round-mode=places,round-precision=3]{1.157876} \\%[-8pt]  
    & $J_{\text{drag}}$ & \num[round-mode=places,round-precision=3]{0.955513} & \num[round-mode=places,round-precision=3]{0.954674} & \num[round-mode=places,round-precision=3]{0.954314} \\%[-3pt]
    \midrule %\\%[-24pt]
    \multirow{3}{*}{$\nu_0 = 0.95$} & $\nu$ & \num[round-mode=places,round-precision=3]{0.950000} & \num[round-mode=places,round-precision=3]{0.949312} & \num[round-mode=places,round-precision=3]{0.949980} \\%[-8pt]  
    & $J_{\sms{E}}$ & \num[round-mode=places,round-precision=3]{0.861774} & \num[round-mode=places,round-precision=3]{0.877232} & \num[round-mode=places,round-precision=3]{0.871795} \\%[-8pt]  
    & $J_{\text{drag}}$ & \num[round-mode=places,round-precision=3]{0.959899} & \num[round-mode=places,round-precision=3]{0.959547} & \num[round-mode=places,round-precision=3]{0.959540} \\%[-3pt] 
    \midrule %\\%[-24pt]
    \multirow{3}{*}{$\nu_0 = 1.00$} & $\nu$ & \num[round-mode=places,round-precision=3]{1.000} & \num[round-mode=places,round-precision=3]{1.000} & \num[round-mode=places,round-precision=3]{1.000} \\%[-8pt]  
    & $J_{\sms{E}}$ & \num[round-mode=places,round-precision=3]{0.500} & \num[round-mode=places,round-precision=3]{0.500} & \num[round-mode=places,round-precision=3]{0.500} \\%[-8pt]  
    & $J_{\text{drag}}$ & \num[round-mode=places,round-precision=3]{1.000} & \num[round-mode=places,round-precision=3]{1.000} & \num[round-mode=places,round-precision=3]{1.000} \\%[-3pt] 
    
    \bottomrule
    \end{tabular}
    \caption{{\em Comparison of a variety of constraints and shapes. The corresponding shape are exhibited in Fig.~\ref{fig:JvsRV}. The Max Efficiency* shapes are obtained from using the Min Drag Force shape as the initial shape.}}\label{tab:compare_all} %All results are calculated under $NL=21, NLuslip=200$ (I need to explain this), starting from the same initial shape. Comparison of swimmer shapes. The first column lists the results of prolate spheroidal swimmers. The second column lists the optimal drag minimization results for various reduced volume $\nu$. The third column lists optimal shapes for maximum swimming efficiency.
\end{table}

%% figure
%\begin{figure}[!h] \centering
%  \includegraphics[width=0.5\textwidth]{figure6.pdf}
%  \caption{Optimization process for $\nu=0.8$. under identical axes scale. The slip velocity profiles are colored on the surface. The slip is scaled to make $U=1$ for comparison.}\label{fig:iterations}
%\end{figure}

The three microswimmers of the same reduced volume $\nu = 0.7$ are shown in Fig.~\ref{fig:flow_ne_08}. The flow fields are obtained by the optimal slip profiles corresponding to these shapes. 
% To get further insights into the optimal shapes, we investigate the flow field around the microswimmers in Fig.~\ref{fig:flow_ne_08}.
All microswimmers are swimming at the unit translational velocity.
In addition to the front-back symmetric shapes, the optimal slip profiles for each of these shapes are also front-back symmetric, expected from the linearity of Stokes equations. 
%The shape and the slip profile combined lead to classical flow fields of neutral swimmers (as opposed to pushers or pullers).
The fluid velocities around the microswimmers are faster closer to the swimmer body because of the slip-boundary conditions, and quickly decay to 0 away from the swimmer.
Compared to the two pointy shapes, the prolate spheroid has a bigger region with high velocities close to the poles. These high-velocity regions are sources of extra power loss (fluid dissipation) that negatively impact swimming efficiencies.

%{\em continue working on it...}

% The prolate spheroids have focused velocities close to the poles, which do not contribute much to the swimming efficiencies.

% Optimization progress is shown in Fig.~\ref{fig:shapeiterations} (body frame).
% The histograms of the velocity magnitudes are shown in Fig.~\ref{fig:shapeiterations2}.
% (Same stuff in the lab frame is shown in the appendix, or not...).
% figure
\begin{figure}[t] \centering
  \includegraphics[width=0.95\textwidth]{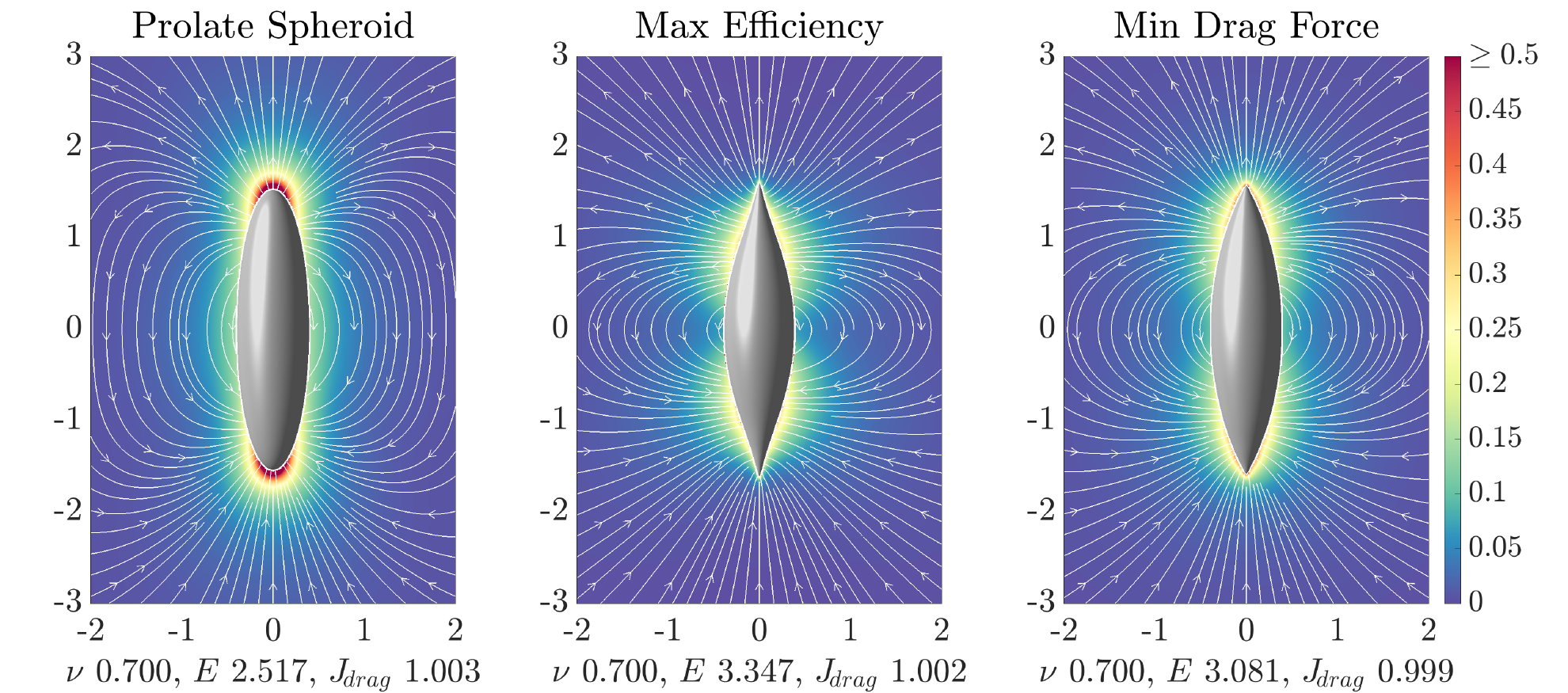}
  \caption{{\em Velocity flow fields (in lab frame) for different body shapes when the reduced volume is $\nu=0.7$. The swimming translation velocity $U$ is scaled to one.}}
  \label{fig:flow_ne_08}
\end{figure}

\section{Conclusions} \label{sc:conclusions}
In this work, we proposed a computational framework that optimizes the shape and the slip velocity of a slip-driven axisymmetric microswimmer. The objective function is chosen to be the swimming efficiency in the fluid medium enclosing the microswimmer, and the microswimmers are subject to constant reduced volumes. The forward problems are solved using a boundary integral equation (BIE) method of high accuracy.

The optimization is performed iteratively during which the optimal slip profile for a given shape and the optimal shape for a given slip profile are updated. We use an Augmented Lagrangian Method (ALM) to enforce the constraint of constant reduced volume.
The slip-optimization is formulated as a symmetric generalized eigenvalue problem that requires only solving one forward problem with mixed boundary condition \eqref{eq:optslip} in addition to the adjoint problem \eqref{eq:adj0} - an improvement of our previous work \cite{guo2021slip} which requires solving one forward problem for each basis function of the slip profile.
The shape sensitivities suitable for this problem are derived using standard treatments. The results are validated against Finite Difference method and show excellent accuracy.

Our optimization results show that the optimal shape for the axisymmetric slip-driven microswimmer with a given reduced volume is a front-back symmetric elongated shape with sharp tips. Optimal slip profiles associated with these shapes result in small high-velocity regions close to the microswimmer. The hydrodynamic efficiency can be significantly higher than that of a prolate spheroid with the same reduced volume $\nu$, especially at small $\nu$'s. It is also worthy to note that the shape that minimizes the fluid drag also out-performs the spheroids at given $\nu$.

We note that the optimal shapes obtained here are different than those in \cite{vilfan2012optimal}, in which the optimal shapes demonstrate ripple-like features for large minimal-curvature constraints and long protrusions from poles for small minimal-curvature constraints. In the latter case, the optimal slip velocity is high close to the tips of the protrusions and roughly uniform over the ``body'' of the swimmer. We show that the optimal swimmer shape is more regular, sharp poles nonetheless, if we allow arbitrarily small curvature. 

\bibliographystyle{siamplain}
\bibliography{ref}

\begin{appendix}

\section{Parametrization using B-splines}\label{apd:bsp} We first denote the $5$-th order B-spline on knots $\{0,1,2,3,4,5,6\}$ by $\mathcal{B}_{0} (t)$. By horizontal shift, let $\mathcal{B}_{k} (t) = \mathcal{B}_{0}(t-k)$, for any integer $k$. For the parameter domain $[0,L]$ ($L$ being either $\pi$ or $2\pi$), the following transformation of $\mathcal{B}_k$ is used,
\begin{equation}
B_k(t) = \mathcal{B}_{k-6}\left(\frac{tN\!_L}{L}\right), \quad k=1,2, \dots, (N_L+5), \quad t\in [0,L]
\label{eq:Bbasis}
\end{equation} 
where $N\!_L$ is the number of uniform subintervals of $[0,L]$. 
%In this work, for $u^{\sms{S}}$, we set $N\!_L=2N_u+2$ and $L=2\pi$. For the generating curve $\gamma$, we set $N\!_L+5=N\!_{R}=N\!_{Z}$ and $L = \pi$. 
The (scalar) slip velocity profile $u^{\sms{S}}$ is used to characterize the slip velocity $\bsb{u}^{\sms{S}}$ on $\Gamma$ with $t\in[0,\pi]$. We obtain $N_u$ (uniform) interior grid points of $u^{\sms{S}}$,
\begin{equation}
\text{(a)} \quad u_{\sss 0} = u^{\sms{S}}(0) = 0,  \;\; u_{\sss N_u+1} =u^{\sms{S}}(\pi) = 0, \quad
\text{(b)} \quad u_{\sss k} = u^{\sms{S}}\!
\left(\frac{k\pi}{N_u+1} \right), \quad k=1, \dots, N_u, \quad \text{on $\gamma$}, \label{eq:samplept}
\end{equation} 
where (\ref{eq:samplept}a) stems from \eqref{eq:us0}. In practice, an extended vector $\bsb{v}_{\sms{ext}}$ for $t\in [0,2\pi]$ is used to fit B-spline interpolation,
\begin{equation}
\bsb{v}_{\sms{ext}}= \tranv{ 0, u_{\scriptscriptstyle 1}, u_{\scriptscriptstyle 2}, \dots, u_{{\scriptscriptstyle N_u -1}}, u_{\scriptscriptstyle N_u}, 0, -u_{\scriptscriptstyle N_u}, -u_{\scriptscriptstyle {N_u}-1}, \dots, -u_{\scriptscriptstyle 2}, -u_{\scriptscriptstyle 1}, 0 }.  \label{eq:ubsp}
\end{equation}
The purpose of $\bsb{v}_{\sms{ext}}$ is to maintain periodicity of derivatives at the poles.
For a given $\bsb{v}_{\sms{ext}}$, a function $w(t)$ is constructed by
\begin{equation}
w(t) %= \sum_{k=1}^{2N_u+7} c_k B_k(t)  
= \bsb{c}\tv \bsb{B}(t) \label{eq:ucB}
\end{equation}
where $\bsb{B}(t) = \tranv{B_{\sss 1}(t), \dots, B_{\sss 2N_u+7}(t)}$ is defined by \eqref{eq:Bbasis}.
The vector $\bsb{c} = \tranv{c_{\sss 1}, \dots, c_{\sss 2N_u+7}}$ is solved by
\begin{equation}
\left\{
\begin{array}{l}
w\left(\frac{\pi(i-1)}{N_u+1} \right) =v_{i}, 
\quad i=1, \dots, 2N_u+3, \\
\displaystyle \frac{d^n w(t)}{dt^n}\Big|_{t=0}= \frac{d^n w(t)}{dt^n} \Big|_{t=\pi}, \quad n=1,2,3,4. \label{eq:fitbsp}
\end{array}
\right.
\end{equation}
where $v_i$ is the $i$-th component of $\bsb{v}_{\sms{ext}}$.

The component function $w_{\sss k}(t)$ in $\bsb{w}(t) = \tranv{w_{\sss 1}(t),\dots,w_{\sss N_u}(t)}$ in \eqref{eq:finiteparamvec} is obtained by letting $k\in\{1,\dots,N_u\}$, $u_{\sss k}=1$ and $u_{\sss n}=0$ when $n\ne k$, and solving for $\bsb{c}$ in \eqref{eq:ucB} by \eqref{eq:fitbsp}.

The representations of $R$ and $Z$ involve $B_k$ directly, thus, we denote $\bsb{B}(t) = \tranv{B_1(t),\dots, B_{N_\gamma}(t)}$ in \eqref{eq:finiteparamvec}. %The interpolations are augmented with $R(0)=R(\pi)=0$ and $Z'(0)=Z'(\pi)=0$. 
An arbitrary $\bsb{\xi}_{\gamma}$ does not satisfy (\ref{eq:rz}a,c), therefore, we determine the value of $\xi^1_{R}$ using $R(0)=0$ by the values of $\xi^2_{R}, \xi^3_{R}, \xi^4_{R}, \xi^5_{R}$, considering the fact that $B_k(0)\ne 0$ only for $k\le 5$. Choosing $N\!_{\gamma}>10$, we similarly determine $\xi^{N\!_{\gamma}}_{R}$ by enforcing $R(\pi)=0$ using the values of $\xi^{N\!_{\gamma}-4}_{R},\xi^{N\!_{\gamma}-3}_{R},\xi^{N\!_{\gamma}-2}_{R},\xi^{N\!_{\gamma}-1}_{R}$ . Similarly, $\xi^1_{Z}$ is determined on values of $\xi^2_{Z}, \xi^3_{Z}, \xi^4_{Z}, \xi^5_{Z}$ by $Z'(0)=0$ and $\xi^{N\!_{\gamma}}_{Z}$ on the values of $\xi^{N\!_{\gamma}-4}_{Z},\xi^{N\!_{\gamma}-3}_{Z},\xi^{N\!_{\gamma}-2}_{Z},\xi^{N\!_{\gamma}-1}_{Z}$ by $Z'(\pi)=0$. These conditions reduce the degree of freedom of $\bsb{\xi}_{\gamma}$ from $2N_{\gamma}$ to $(2N_{\gamma} - 4)$.

\section{Auxiliary proofs}
\label{apd:proofs}

\proofstep{Proof of Lemma~\ref{lemma:Etensor}} A straightforward derivation shows that $\nabla\cdot \bsb{E} \big( (\bsb{u},p), (\bsb{u},p)\big) = \bsb{0}$ holds for any $(\bsb{u}, p)$ satisfying \eqref{eq:forward} (use component notation and verify that $\sum_{j}\pd_j E_{ij} = 0$, $i=1,2,3$). We consequently have
\begin{equation}
\bsb{E}\big( (\bsb{u}, p), (\bsb{u}, p) \big) : \nabla\tv\bsb{\theta} = \nabla\cdot \big[ \bsb{E}\big( (\bsb{u}, p), (\bsb{u}, p) \big) \cdot \bsb{\theta} \big] - \big[\nabla\cdot  \bsb{E}\big( (\bsb{u}, p), (\bsb{u}, p) \big) \big]\cdot \bsb{\theta} = \nabla\cdot \big[ \bsb{E}\big( (\bsb{u}, p), (\bsb{u}, p) \big) \cdot \bsb{\theta} \big]
\end{equation}
Then, observing that $\big( (\bsb{u}, p), (\bsb{v}, q) \big) \mapsto \bsb{E}\big( (\bsb{u}, p), (\bsb{v}, q) \big)$ defines a symmetric bilinear form, we invoke the polarization identity and obtain
\begin{equation}
\begin{aligned}
4\bsb{E}\big( (\bsb{u}, p), (\hat{\bsb{u}}, \hat{p}) \big) : \nabla\tv\bsb{\theta}
 &= \big[ \bsb{E}\big( (\bsb{u}\!+\!\hat{\bsb{u}}, p\!+\!\hat{p}), (\bsb{u}\!+\!\hat{\bsb{u}}, p\!+\!\hat{p}) \big) -  \bsb{E}\big( (\bsb{u}\!-\!\hat{\bsb{u}}, p\!-\!\hat{p}), (\bsb{u}\!-\!\hat{\bsb{u}}, p\!-\!\hat{p}) \big) \big] : \nabla\tv\bsb{\theta} \\[1ex]
 &= \nabla\cdot \big[ \bsb{E}\big( (\bsb{u}\!+\!\hat{\bsb{u}}, p\!+\!\hat{p}), (\bsb{u}\!+\!\hat{\bsb{u}}, p\!+\!\hat{p}) \big)\cdot \bsb{\theta} -  \bsb{E}\big( (\bsb{u}\!-\!\hat{\bsb{u}}, p\!-\!\hat{p}), (\bsb{u}\!-\!\hat{\bsb{u}}, p\!-\!\hat{p}) \big)\cdot\bsb{\theta} \big] \\[1ex]
 &= 4\nabla\cdot \big[ \bsb{E}\big( (\bsb{u}, p), (\hat{\bsb{u}}, \hat{p}) \big)\cdot \bsb{\theta} \big]
\end{aligned}
\end{equation}
whereupon applying the first Green identity (divergence theorem) completes the proof of the claimed identity.

\proofstep{Proof of formula~(\ref{eq:taustar})}
Let the parametric representation of $\Gamma_{\eta}(\bsb{\theta})$ be of the form~\eqref{eq:xeta}. We seek the derivative of the unit tangent vector on $\Gamma_{\eta}$, given by
\begin{equation}
  \alpha_{\eta}(t,\phi)\bsb{\tau}_{\eta}(t,\phi) = \pd_{s}\bsb{x}_{\eta}(t,\phi). \label{eq:taugammaeta}
\end{equation}
with respect to $\eta$ and at $\eta=0$. Since $\bsb{\theta}=\pd_{\eta} \bsb{x}_{\eta}(t,\phi)$, we have
\begin{equation}
  \pd_{\eta}\alpha_{\eta}(t,\phi)\,\big|_{\eta=0}
 = \frac{\pd_t\bsb{x}_{\eta}(t,\phi) \cdot \pd_{\eta s}\bsb{x}_{\eta}(t,\phi) }{\alpha_{\eta}(t,\phi)} \Big|_{\eta=0}
 = [\bsb{\tau}\cdot\pd_{s}\bsb{\theta}](t,\phi).
% = \theta_{\tau}'(t) - \kappa(t) \theta_n(t).
\label{eq:petapsx}
\end{equation}
The derivative $\overstar{\bsb{\tau}}:=\partial_{\eta}\bsb{\tau}|_{\eta=0}$ is hence evaluated from \eqref{eq:taugammaeta}, \eqref{eq:petapsx} and \eqref{eq:frenet} as
\begin{equation}
  \overstar{\bsb{\tau}}(t,\phi)
 = \pd_{\eta} \bsb{\tau}_{\eta}(t,\phi) \Big|_{\eta=0}
 = \tfrac{1}{\alpha(t)} \big( \pd_t\bsb{\theta} - \big[ \bsb{\tau}\cdot\pd_{s}\bsb{\theta} \big] \bsb{\tau} \big)(t,\phi)
 = \tfrac{1}{\alpha(t)} \big( [ \bsb{n}\cdot\pd_{s}\bsb{\theta} ] \bsb{n} \big)(t,\phi)
 = \big[ \kappa \theta_{\tau} + \tfrac{1}{\alpha} \theta_n' \big](t)\, \bsb{n}(t,\phi)
\end{equation}
which completes the proof of~\eqref{eq:taustar}.

\section{Differential operators using curvilinear coordinates and proof of Lemma~\ref{lemma:grad}}\label{apd:proofprop} Let points $\bsb{x}$ in a tubular neighborhood $V$ of $\Gamma$ be given in terms of curvilinear coordinates $(t,h)$, so that
\begin{equation}
\bsb{x} = \bsb{x}(t,\phi) + h \bsb{n}(t,\phi),
\end{equation}
with $\bsb{x}(t,\phi)$ and $\bsb{n}(t,\phi)$ as given in \eqref{wall:shape} and \eqref{tn:axi}, respectively, and let
\begin{equation}
\bsb{v}(\bsb{x}) = v_{\tau}(t,h)\bsb{\tau}(t,\phi) + v_n(t,h) \bsb{n}(t,\phi) \label{eq:vecv}
\end{equation}
denote a generic axisymmetric vector field in $V$. Then, at any point $\bsb{x}=\bsb{x}(t,\phi)$ on $\Gamma$ (i.e., at $h=0$), we have
\begin{equation}
\begin{aligned}
\nabla \bsb{v}
 &= (\tfrac{1}{\alpha}\pd_t v_ s -\kappa v_n) \bsb{\tau} \otimes \bsb{\tau}
  + (\tfrac{1}{\alpha}\pd_t v_n + \kappa v_{\tau}) \bsb{n} \otimes \bsb{\tau}
  + \tfrac{1}{\alpha R} (R'v_{\tau}+Z'v_n) \bsb{\nu} \otimes \bsb{\nu}
  + \pd_hv_{\tau}  \bsb{\tau} \otimes \bsb{n} + \pd_hv_n \bsb{n} \otimes \bsb{n}\\
  \nabla\cdot\bsb{v} &= \tfrac{1}{\alpha}\pd_t v_{\tau} - \kappa v_n + \tfrac{1}{\alpha R}(R' v_{\tau} + Z' v_n) + \pd_h v_n,
\end{aligned}
\end{equation}
where $\bsb{\nu} = \sin\phi \bsb{e}_x - \cos\phi \bsb{e}_y = \bsb{n}\times\bsb{\tau}$. In particular, the transformation velocity $\bsb{\theta}$ being of the form~\eqref{eq:vecv}, we have
\begin{equation}
  \divs \bsb{\theta} = \tfrac{1}{R\alpha} \big[ \pd_t (R\theta_{\tau}) + (Z'-\kappa R)\theta_n \big]. \label{divS:def}
\end{equation}
Assuming incompressibility, the condition $\nabla\cdot \bsb{v} = 0$ can be used for eliminating $\pd_h v_n$ and we obtain
\begin{equation}
  \nabla\bsb{v}
 = (\tfrac{1}{\alpha}\pd_t v_ s -\kappa v_n)(\bsb{\tau} \otimes \bsb{\tau} -\bsb{n} \otimes \bsb{n} )
 + (\tfrac{1}{\alpha}\pd_t v_n + \kappa v_{\tau}) \bsb{n} \otimes \bsb{\tau} + \pd_h v_{\tau}  \bsb{\tau} \otimes \bsb{n}
 + \frac{1}{\alpha R}(R' v_{\tau} + Z' v_n)(\bsb{\nu}\otimes\bsb{\nu}-\bsb{n} \otimes \bsb{n}).
\end{equation}

For the forward solution, we have $\bsb{u}=U \bsb{e}_z + u^{\sms{S}} \bsb{\tau}=\big(u^{\sms{S}}+\tfrac{1}{\alpha}UZ'\big)\bsb{\tau}-\tfrac{1}{\alpha} UR'\bsb{n}$ on $\Gamma$. Recalling that $2\bsb{D}[\bsb{u}] = \nabla\bsb{u}+\nabla\tv\bsb{u}$ and using $\kappa=\tfrac{1}{\alpha^3}(Z'R''-R'Z'')$ and $R'^2+Z'^2=\alpha^2$, we obtain
\begin{equation} 
\begin{aligned}
  \nabla\bsb{u}
 &= \frac{(u^{\sms{S}})'}{\alpha} \bsb{\tau} \otimes \bsb{\tau} + \frac{R'u^{\sms{S}}}{R\alpha}\bsb{\nu}\otimes\bsb{\nu}
 - \frac{(Ru^{\sms{S}})'}{R\alpha} \bsb{n} \otimes \bsb{n} + \kappa u^{\sms{S}}\bsb{n}\otimes\bsb{\tau}
 + \pd_h u_{\tau} \bsb{\tau} \otimes \bsb{n} \\
 &= \frac{(Ru^{\sms{S}})'}{R\alpha} \big( \bsb{\tau} \otimes \bsb{\tau} - \bsb{n} \otimes \bsb{n} \big)
  + \frac{R'u^{\sms{S}}}{R\alpha} \big( \bsb{\nu}\otimes\bsb{\nu} - \bsb{\tau} \otimes \bsb{\tau} \big)
  + \kappa u^{\sms{S}}\bsb{n}\otimes\bsb{\tau} + \pd_h u_{\tau} \bsb{\tau} \otimes \bsb{n}, \\
  2\bsb{D}[\bsb{u}]
 &= \frac{(2Ru^{\sms{S}})'}{R\alpha} \big( \bsb{\tau} \otimes \bsb{\tau} - \bsb{n} \otimes \bsb{n} \big)
  + \frac{2R'u^{\sms{S}}}{R\alpha} \big( \bsb{\nu}\otimes\bsb{\nu} - \bsb{\tau} \otimes \bsb{\tau} \big)
  + \big(\pd_h u_{\tau} + \kappa u^{\sms{S}}\big) \big(\bsb{n}\otimes\bsb{\tau} +\bsb{\tau} \otimes \bsb{n}\big),
\end{aligned}
\label{eq:graduand2Du0}
\end{equation}
on $\Gamma$. The stress tensor $\bsb{f} = -p\bsb{n}+2\mu\bsb{D}[\bsb{u}]\cdot\bsb{n}$ on $\Gamma$ is then found as
\begin{equation}
\bsb{f} = - \big( p + \frac{2\mu}{R\alpha}(Ru^{\sms{S}})' \big)\bsb{n} + \mu (\pd_h u_{\tau} + \kappa u^{\sms{S}}) \bsb{\tau}. \label{eq:flong}
\end{equation}
In particular, taking the tangential and normal projections of $\bsb{f}$, we obtain
\begin{equation}
  \frac{2\mu}{R\alpha}(Ru^{\sms{S}})' = -f_n-p, \qquad \pd_h u_{\tau} = \frac{1}{\mu}f_{\tau} - \kappa u^{\sms{S}}
\end{equation}
which, used in~\eqref{eq:graduand2Du0}, establishes the first part of the lemma. Using $(u^{\sms{S}},U)=(0,1)$ in this result then yields the second part.

\end{appendix}

\end{document}